\documentclass{article}
\usepackage{graphicx} %

\usepackage{latexsym,amssymb,amsmath,amsfonts,amsthm}
\numberwithin{equation}{section}
\usepackage{enotez,hyperref,float}
\usepackage{forest}
\usepackage{thmtools, thm-restate}

\usetikzlibrary{arrows.meta, positioning}
\usetikzlibrary{decorations.markings}

\usepackage{color}
\usepackage{epsfig}
\usepackage{enumerate}

\usepackage{svg}

\newtheorem{theorem}{Theorem}

\newtheorem{corollary}[theorem]{Corollary}

\newtheorem{lemma}{Lemma}

\graphicspath{ {./assets/} }

\newtheorem{proposition}{Proposition}
\newtheorem*{proposition*}{Proposition}
\newtheorem{remark}{Remark}

\newcommand{\EE}{{\mathbb{E}}}

\DeclareMathOperator*{\argmin}{arg\,min}

\newcommand{\Cov}{\mathrm{Cov}}

\newcommand{\gsimas}[2]{\overset{{#1}\to {#2}}{\gtrsim}}
\newcommand{\lsimas}[2]{\overset{{#1}\to {#2}}{\lesssim}}
\newcommand{\simas}[2]{\overset{{#1}\to {#2}}{\sim}}
\newcommand{\ggas}[2]{\overset{{#1}\to {#2}}{\gg}}
\newcommand{\llas}[2]{\overset{{#1}\to {#2}}{\ll}}
\newcommand{\ordas}[2]{\overset{{#1}\to {#2}}{\asymp}}

\def\beq{ \begin{equation} }
\def\eeq{ \end{equation} }

\def\square{\vcenter{\vbox{\hrule height .4pt
  \hbox{\vrule width .4pt height 5pt \kern 5pt
        \vrule width .4pt} \hrule height .4pt}}}

\def\sqz{\kern-0.2em}

\usepackage{graphicx}
\graphicspath{%
    {converted_graphics/}%
    {/}%
}

\usepackage[
backend=biber,
style=numeric,
sorting=none
]{biblatex}
\usepackage{authblk}

\title{A Central Limit Theorem for the Driver Site Frequency Spectrum}
\author[1,2,3]{Armaan Ahmed}
\author[4]{Jasmine Foo}
\author[5]{Kevin Leder}

\affil[1]{Department of Applied Mathematics \& Statistics, Johns Hopkins University}
\affil[2]{Department of Mathematics, Johns Hopkins University}
\affil[3]{Department of Statistics, University of California, Berkeley}
\affil[4]{Department of Mathematics, University of Minnesota}
\affil[5]{Department of Industrial and Systems Engineering, University of Minnesota}

\usepackage{nicematrix}

\usepackage{lineno}

\providecommand{\classification}[2]
{
  \small	
  \textbf{\textit{#1---}} #2
}

\begin{document}

\maketitle
\begin{abstract}
  We consider a supercritical two-type continuous-time linear birth-death process with mutation and selection, in which wild-type individuals give rise to mutant offspring with a larger net-growth rate. We previously had obtained first-order law of large number results for the driver site frequency spectrum. We return to obtain a second-order central limit theorem, in which we show stable convergence of a scaling of the deviation of the SFS from its first-order limit. We provide a central limit theorem at both large detection times and sizes. Then, we apply these results to obtain confidence intervals for a class of consistent estimators for the fitness advance.
\end{abstract}

\classification{Keywords}{Site Frequency Spectrum, Birth-Death Process, Central Limit Theorem, Cancer Evolution}

\classification{MSC2020}{Primary: 60J80. Secondary: 92D10, 92D25, 60F05}

\section{Introduction}

Cells can accumulate mutations in a sequential manner, through which it becomes cancerous \cite{bielas2006,hanahan2022}. One way to summarize the data concerning the frequency of these mutations is through the site frequency spectrum (denoted $(S_j(t))_{j=1}^\infty$). $(S_j(t))_{j=1}^\infty$ is a random vector with the $j$th coordinate denoting the number of mutations present in $j$ cells at time $t$. These mutations can either be deleterious, neutral, or beneficial. Here, we consider the setting of driver (i.e. beneficial) mutations which is relevant in studying the development of cancer, as in our previous work \cite{ahmed2026}.

Specifically, the model we consider is a two-type continuous-time linear birth-death process, wherein type-0 cells can mutate to type-1 cells with a fitness increase in its birth rate. Under this model, we had previously obtained a law of large numbers concerning $S_j(t)$ and an associated estimator for the fitness increase of the type-1 clone over the type-1 clone \cite{ahmed2026}. Here, we provide second-order asymptotics, or a central limit theorem, for the SFS and use it to construct confidence intervals for our estimator.

Prior works have shown first-order-type results for the SFS under deleterious \cite{cheek2018}, neutral \cite{durrett2013,bozic2016,williams2016,cheek2018,gunnarsson2025,schweinsberg2025}, or selective/rescue \cite{cheek2018,bonnet2024,ahmed2026} evolution. Second-order results are somewhat harder to come-by. By exploiting recently-developed coalescent tools for a sample taken from a Galton-Watson process \cite{lambert2018,harris2020}, Schweinsberg and Shuai \cite{schweinsberg2025} have derived a central limit theorem for a sample of the SFS in the neutral mutation setting. They use a backwards-in-time approach exploiting the coalescent structure of a sample taken from a Galton-Watson tree \cite{lambert2018}, whereas in this study we use a different forward-in-time approach to obtain a central limit theorem.

\subsection{Outline}

In Section \ref{sec:clt}, we prove one of our main results concerning the convergence of the scaled deviations of $S_j(t)$. We also examine the large time deviations when $j=j(t)$ depends on $t$. In Section \ref{sec:guass_process}, we provide a law of iterated logarithm result for this Gaussian process. We extend the CLT when detecting at detection sizes in Section \ref{sec:clt_detection}. Finally in Section \ref{sec:estimator_ci}, we provide an application of our results to constructing confidence intervals for an estimator for the fitness advance constructed in our previous work \cite{ahmed2026}. Most proofs are deferred to Section \ref{sec:proofs}.

\section{Model Description} 
Consider a continuous-time Markov branching process with two cell types, denoted type-0 and type-1.
Each type-0 cell independently divides at exponential rate $a$ and dies at rate $d$,  and the total type-0 population count at time $t$ is denoted $Z_0(t)$. Each type-1 cell  divides at rate $a+b$ and die at rate $d$, where $b>0$ reflects a selective advantage. The total type-1 population count at time $t$ is denoted $Z_1(t)$.
Type-1 cells are created via driver mutation acquisition. The total population is denoted $Z(t)=Z_0(t)+Z_1(t)$ 

Type-0 cells produce type-1 cells through two mechanisms, (i) {\it replication-dependent:} at each type-0 division, precisely one offspring is type-1 with probability $v$ and type-0 with probability $1-v$\footnote{Here, we are making a simplification and ignoring the possibility that both type-0 offspring mutate to type-1 cells. This scenario occurs with probability $\mathcal{O}(v^2)$.}, and (ii) {\it replication-independent:} each type-0 cell mutates to type-1 at rate $\mu$, independently of division \cite{abascal2021}. Thus, the effective birth rate of type-0 cells is $a(1-v)$ and the net growth rate is $\lambda_0 \equiv a(1-v)-d-\mu$. The net growth rate of type-1 cells is $\lambda_1 \equiv a+b-d$. Throughout, we assume supercriticality of the type-$0$ population, namely $\lambda_0>0$. The non-extinction event of the type-0 population is denoted $\Omega_0^\infty=\{\forall t:Z_0(t)>0\}$. It is well known that $\mathbb{P}(\Omega_0^\infty):=q_0=1-p_0$, where $p_0=(d+\mu)/(a(1-v))$. The probability of extinction of the type-1 population, conditional on starting with a single type-1 cell, is $p_1=d/(a+b)$. We also let $q_1=1-p_1$ denote its probability of non-extinction.

We denote the transition probabilities of a single type-1 clone by $P_j(t)=\mathbb{P}(Z_1(t)=j|Z_1(0)=i)$.
Let $S_j(t)$ denote the number of type-1 driver clones of size $j$ at time $t$; we refer to $S_j(t)$  as the `driver SFS.' Specifically, letting $Z_1^{(i)}(s)$ denote the size of the $i$th type-1 clone $s$ units after it is initiated:
\[S_j(t)=\sum_{i:T_i\leq t}\mathbf{1}_{Z_1^{(i)}(t-T_i)=j}\]
We also denote $M_t=\#\{i:T_i\leq t\}$ as the number of mutation events from type-$0$ to type-$1$ that have occurred in the interval $[0,t]$. We explore the SFS at both large detection times as well as large detection times $\tau_n=\inf\{t\geq 0:Z(t)\geq n\}$.

It is well-known that $\lim_{t\to\infty}e^{-\lambda_0 t}Z_0(t)$ converges almost surely and in $L^2$ to a limiting random variable $Y$.  The limiting random variable $Y$ is 0 with probability $p_0$, and on the event of survival it is exponential with rate $q_0$ (for example see Theorem 1 in \cite{durrettbranching2015}). Also, $W=\lim_{t\to\infty}e^{-\lambda_1 t}Z(t)=\lim_{t\to\infty}e^{-\lambda_1}Z_1(t)$ almost surely (see Lemma \ref{lem:W_form}).

We will be showing stable convergence with respect to certain $\sigma$-algebras, and so we make reference to certain important $\sigma$-filtrations

\begin{itemize}
    \item Let $\mathcal{F}_t^{\leq i}$ for $i\in \{0,1\}$ denote the smallest $\sigma$-algebra generated either (if $i=0$) by the type-0 process in the time interval $[0,t]$ or (if $i=1$) by the type-0 and type-1 process in the time interval $[0,t]$. The $i=0$ case is most often used, so that we also write $\mathcal{F}_t:=\mathcal{F}_t^{\leq 0}$. We also let $\mathcal{F}_\infty^{\leq i}=\sigma\left(\bigcup_{t\geq 0}\mathcal{F}_t^{\leq i}\right)$. Once again, for shorthand we write $\mathcal{F}_\infty:=\mathcal{F}_\infty^{\leq 0}$.
    \item Let $\mathcal{B}_t^{\leq i}$ for $i\in \{0,1\}$ denote the smallest $\sigma$-algebra containing $\mathcal{F}_t^{\leq i}$ as well as the timings of the type-1 mutation times that occur within the interval $[0,t]$. We also write $\mathcal{B}_t:=\mathcal{B}_t^{\leq0}$. We also let $\mathcal{B}_\infty^{\leq i}=\sigma\left(\bigcup_{t\geq 0}\mathcal{B}_t^{\leq i}\right)$. Once again, for shorthand we write $\mathcal{B}_\infty:=\mathcal{B}_\infty^{\leq 0}$.
    \item Let $\mathcal{I}_t$ be the smallest $\sigma$-algebra containing $\mathcal{B}_t^{\leq 1}$ as well as the identities of the clones to which each type-1 cell belongs to. Finally, define $\mathcal{I}_\infty=\sigma\left(\bigcup_{t\geq 0}\mathcal{I}_t\right)$.
\end{itemize}
Clearly, $\mathcal{F}_t^{\leq i}\subseteq \mathcal{B}_t^{\leq i}\subseteq \mathcal{I}_t$. $(Z_0(t),Z_1(t))$ is adapted to $\mathcal{F}_t^{\leq 1}$. $(Z_0(t),Z_1^B(t))$ and $S_j(t)$ are adapted to $\mathcal{I}_t$. For a sub-$\sigma$-algebra $\mathcal{G}$ of our ambient probability space, $\mathcal{G}$-stable convergence of $\{Q_t\}_t$ to $Q_\infty$ means that for all bounded $\mathcal{G}$-measurable $Q'$ and continuous, bounded $f$:
\begin{align*}
    \mathbb{E}[Q'f(Q_t)]\to \mathbb{E}[Q'f(Q_{\infty})].
\end{align*}

Letting $\ell\in \mathbb{R}\cup \{\pm \infty\}$ and $g,h:\mathbb{R}\to \mathbb{R}_{\geq 0}$, we also use the following asymptotic notation:

\begin{align*}
    g\simas{t}{\ell}h&\Longleftrightarrow \lim_{t\to\ell}\frac{g(t)}{h(t)}=1\\
    g\lsimas{t}{\ell} h&\Longleftrightarrow \limsup_{t\to\ell}\frac{g(t)}{h(t)}<\infty\\
    g\gsimas{t}{\ell} h&\Longleftrightarrow h\lsimas{t}{\ell} g\\
    g\llas{t}{\ell} h&\Longleftrightarrow \lim_{t\to \ell}\frac{g(t)}{h(t)}=0\\
    g\ggas{t}{\ell} h&\Longleftrightarrow h\llas{t}{\ell} g\\
    g\ordas{t}{\ell} h&\Longleftrightarrow g\lsimas{t}{\ell} h\text{ and }g\gsimas{t}{\ell} h
\end{align*}

Additionally, the Landau-O notation is used to specify in-line error terms. That is, letting $g:\mathbb{R}\to \mathbb{R}$ and $h:\mathbb{R}\to \mathbb{R}_{\geq 0}$ be real-functions, we say:
\begin{align*}
    g=\mathcal{O}_{t\to \ell}(h)&\Longleftrightarrow |g|\lsimas{t}{\ell}h,\\
    g=o_{t\to \ell}(h)&\Longleftrightarrow |g|\llas{t}{\ell}h.
\end{align*}

\section{Results}
\subsection{Central Limit Theorem of the SFS}\label{sec:clt}
Our first result is a weak convergence result for the net of centered and scaled SFS vectors in the large time limit. This result is strengthened to establish $\mathcal{I}_\infty$-stable convergence in order to later prove a weak convergence result in the large detection size limit.
\begin{restatable}{theorem}{CLTProduct}\label{thm:clt}
  Let:
    \[X_{t,j}=e^{-\lambda_0 t/2}\left(S_j(t)-(av+\mu)Y\int_0^t P_j(t-s)e^{\lambda_0 s}\,ds\right).\]
    Then:
    \begin{enumerate}
        \item \begin{align*}
        (X_{t,j})_{j=1}^\infty|\Omega_0^\infty\overset{t\to\infty}{\Rightarrow}(X_{\infty,j})_{j=1}^\infty|\Omega_0^\infty
    \end{align*}
    $\mathcal{I}_\infty$-stably, treating $(X_{t,j})_{j=1}^\infty\in \ell^p$ for $t\in [0,\infty]$ and $p\in [2,\infty]$. 
    Conditional on $Y$, $(X_{\infty,j})_{j=1}^\infty$ is a mean zero Gaussian process with covariance function $K(j,k|Y)=\mathbb{E}[X_{\infty,j}X_{\infty,k}|Y]$ (where the Gaussian process is independent of $\mathcal{I}_\infty$). The conditional covariance function is given by $K(j,k|Y)=Y\sigma_{j,k}$, with $\sigma_{j,k}$ defined as in Lemma \ref{eq:sigmajk_def}. Asymptotically if $j\neq k$:
    \begin{align}
        \sigma_{j,k}=\frac{C}{jk}\left(\frac{1}{(j+k)^{\lambda_0/\lambda_1}}+\mathcal{O}_{j+k\to \infty}(1/(j+k)^{\lambda_0/\lambda_1+1})\right)\label{eq:sigma_asymptotic}
    \end{align} 
    and along the diagonal
    \begin{align}
        \sigma_{j,j}=\frac{C_2}{j^{\lambda_0/\lambda_1+1}}+\mathcal{O}_{j\to\infty}\left(\frac{1}{j^{\lambda_0/\lambda_1+2}}\right)
    \end{align}
    for explicitly determinable constants $C,C_2$.
        \item\begin{align}
    \{X_{t,\geq 1}\}\cup (X_{t,j})_{j=1}^\infty|\Omega_0^\infty\overset{t\to\infty}{\Rightarrow}\{X_{\infty,\geq 1}\}\cup (X_{\infty,j})_{j=1}^\infty|\Omega_0^\infty\label{eq:extended_clt}
\end{align}
    $\mathcal{I}_\infty$-stably, in $\ell^p$ for $p\in [2,\infty]$. The limit vector, conditional on $Y$, is a 0-mean Gaussian process with covariance kernel $\mathbb{E}[X_{\infty,i}X_{\infty,j}|Y]=Y\sigma_{i,j}$, $\mathbb{E}[X_{\infty,i}X_{\infty,\geq j}|Y]=Y\sum_{k=j}^\infty\sigma_{i,k}$, and $\mathbb{E}[X_{\infty,\geq i}X_{\infty,\geq j}|Y]=Y\sum_{k_1=i,k_2=j}\sigma_{k_1,k_2}$.
    \end{enumerate}
        
\end{restatable}

\begin{proof}
    See Section \ref{app:clt_time}
\end{proof}

Point (2) is a useful extension of Point (1) as it can describe convergence of the scaled deviations of the SFS greater than certain sizes. In particular, using that $X_{\leq j}=\sum_{k=1}^j X_k$ and that $\ell^p \ni \{x_{\geq 1}\}\cup (x_k)_{k=1}^\infty\to (x_{\geq 1},\sum_{k=1}^j x_k)$ is continuous, we obtain from \ref{eq:extended_clt} that for all $j$:
\begin{align*}
    \{X_{t,\geq 1}\}\cup \{X_{t,\leq j}\}\overset{t\to\infty}{\Longrightarrow} \{X_{\infty,\geq 1}\}\cup \{X_{\infty,\leq j}\}
\end{align*}
$\mathcal{I}_\infty$-stably, in $\mathbb{R}^2$. This is useful in describing the asymptotic variance for the class of estimators of the fitness advance $\{\widehat{b}_{\leq j}(\cdot)\}_{j=1}^\infty$ which we describe in Section \ref{sec:estimator_ci}, as for each $j$, $\widehat{b}_{\leq j}(\cdot)$ depends on the proportion of mutations present at frequency $\leq j$. 

{\bf\em Qualitative comparison of our proof to that of Schweinsberg and Schuai \cite{schweinsberg2025}} Our proof relies on breaking up the the driver SFS into contributions from each type-0 lineage at large time $T$. In detail, we can decompose:
\begin{align}
    X_{t,j}=e^{-\lambda_0T/2}\sum_{i=1}^{Z_0(T)}X_{t-T,j}^{T,(i)}+o_{t\to\infty}(1),\label{eq:fundamental_decomposition}
\end{align}
where $(X_{t-T}^{T,(i)})_{t\geq T}|\mathcal{I}_T\overset{d}{=}(X_t)_{t\geq 0}$ on $\{i\leq Z_0(T)\}$ and $(X_{t-T}^{T,(i)})_{t\geq T}|\mathcal{I}_T$ are i.i.d. over $i\in \{1,\dots,Z_0(T)\}$. As $\exp(\lambda_0T)\asymp Z_0(T)$ as $T\to\infty$ on $\Omega_0^\infty$, so \ref{eq:fundamental_decomposition} is a standardized sum of i.i.d. deviations. Therefore, similar arguments as used in the proof of the central-limit-theorem proves convergence in finite-dimensional distributions. A tightness argument thus gives Point 1 of Theorem \ref{thm:clt}. This argument quite easily extends when also considering fluctuations in the number of clones above a certain size, yielding Point 2.

This proof structure differs fundamentally from that of Schweinsberg and Schuai \cite{schweinsberg2025}, who investigate the SFS of a sample of a neutrally growing continuous-time birth-death process. They note that the length of the ancestral branches in the sampled population supporting $k$ precisely leaves can be obtained from the coalescence times. Thus, by using the explicit formulas for the joint density of these coalescence times available in \cite{lambert2018}, they obtain limit theorems for these branch lengths. Since branch lengths are related to the SFS (assuming that mutations accrue in a replication-dependent manner in our terminology), limit theorems for the SFS can thus be devised. Therefore, we take a forward-in-time approach whereas Schweinsberg and Schuai take a backwards-in-time approach.

{\bf\em The case of increasing $j(\cdot)$} And, we can consider fluctuations of the SFS if we let $j$ vary with $t$. Interestingly in this case, the limiting variance will depend on the expectation of (a scaling of) $S_{j(t)}(t)$.

\begin{restatable}{proposition}{CLTJt}\label{prop:CLT_Jt}
    Let $j(t)\ll \exp(t\lambda_0\lambda_1/(\lambda_0+\lambda_1))$ such that $j$ is monotone increasing and $j(t)\to\infty$ as $t\to\infty$. Then:
    \[
    e^{-t\lambda_0/2}j(t)^{\frac{\lambda_1+\lambda_0}{2\lambda_1}}\left(S_{j(t)}(t)-(av+\mu)Y\int_0^t P_{j(t)}(t-s)e^{\lambda_0 s}\,ds\right)\Rightarrow X_{\infty,j(\cdot)}
    \]
    $\mathcal{I}_\infty$-stably. Conditional on $Y$, $X_{\infty,j(\cdot)}$ is distributed as a 0-mean Gaussian with random variance $Y\cdot L$. Here, $L$ is the same as $L_{j(\cdot)}$ from Proposition 5, Point 3 of \cite{ahmed2026}:
    \[L:=\lim_{t\to\infty}e^{-\lambda_0 t}j(t)^{(\lambda_0+\lambda_1)/\lambda_1}\mathbb{E}[S_{j(t)}(t)].\]
\end{restatable}
\begin{proof}
    See Section \ref{sec:easy_cors}.
\end{proof}

\subsection{Deviations in the limiting Gaussian process}\label{sec:guass_process}

We notice that the correlation structure of the limiting Gaussian process is asymptotically very similar to that of a sequence of independent Gaussians. 

\begin{restatable}{corollary}{CorrelationStructure}\label{cor:clt}
    Let $K(j,k)$ be the covariance function of the infinite Gaussian mixture and let $\rho(j,k):=K(j,k)/\sqrt{K(j,j)K(k,k)}$ denote the correlation function. Let $k(j)\neq j$ for all large $j$, then $\rho(j,k(j))\to 0$ as $j\to\infty$.
\end{restatable}
\begin{proof}
    See Section \ref{sec:easy_cors}
\end{proof}

Thus, it is not surprising to find the following large-deviation behavior in $j$ for the limiting mixture of $Y$ with the Gaussian process examined in Theorem \ref{thm:clt}.:
\begin{restatable}{theorem}{DeviationLargeJ}\label{thm:deviation_large_j}
    We have the following large $j$ behavior:
    \begin{align*}
        &\limsup_{j\to\infty}\frac{X_{\infty,j}}{\sqrt{2Y\sigma_{j,j}\log j}}=1\\
        &\liminf_{j\to\infty}\frac{X_{\infty,j}}{\sqrt{2Y\sigma_{j,j}\log j}}=-1\
    \end{align*}
    almost surely.
\end{restatable}
\begin{proof}
    See Section \ref{app:deviation_limiting}.
\end{proof}

\subsection{CLT of the SFS at large detection sizes}\label{sec:clt_detection}

We are also able to get a CLT along the sequence of detection times. This is relevant in settings in which a population is detected at a certain size.

\begin{restatable}{theorem}{CLTDetectionSize}\label{thm:clt_detection_size}
Define:
\begin{align*}
    \overline{X}_{n,j}=n^{-\frac{\lambda_0}{2\lambda_1}}\left(S_j(\tau_n)-(av+\mu)Y\int_0^{\tau_n}P_j(\tau_n-s)e^{\lambda_0 s}\,ds\right)_{j=1}^\infty
\end{align*}
\begin{enumerate}
    \item \[(\overline{X}_{n,j})_{j=1}^\infty|\Omega_0^\infty\overset{n\to\infty}{\Rightarrow}W^{-\frac{\lambda_0}{2\lambda_1}}(X_{\infty,j})_{j=1}^\infty|\Omega_0^\infty\]
    $\mathcal{I}_\infty$-stably, in $\ell^p$ for $p$ sufficiently large.
    \item \[\{\overline{X}_{n,\geq 1}\}\cup(\overline{X}_{n,j})_{j=1}^\infty|\Omega_0^\infty\overset{n\to\infty}{\Rightarrow}W^{-\frac{\lambda_0}{2\lambda_1}}(X_{\infty,j})_{j=1}^\infty|\Omega_0^\infty\]
    $\mathcal{I}_\infty$-stably, in $\ell^p$ for $p$ sufficiently large.
\end{enumerate}
\end{restatable}
\begin{proof}
    See Section \ref{app:clt_detection_size}.
\end{proof}

The proof uses a similar decomposition \ref{eq:fundamental_decomposition} as in Theorem \ref{thm:clt} but preceeding this step, the random time $\tau_n$ is approximated by $\tau_n':=\inf\{t>0:Z_1^{[0,T_n]}(t)=n\}$. $Z_1^{[0,T_n]}(t)$ counts the sizes at time $t$ of all type-1 clones produced in the interval $[0,T_n]$. Thus, $\tau_n'$ is the time at which the early clones of the population reaches size $n$. For specially chosen $T_n\to \infty$ as $n\to\infty$:
\[X_{\tau_n,j}\approx X_{\tau_n',j}=e^{-\lambda_0 T_n/2}\sum_{i=1}^{Z_0(T_n)}X_{\tau_n'-T_n,j}^{T_n,(i)}+o_{n\to\infty}(1).\]
Importantly, $\tau_n'$ is $\mathcal{I}_{T_n}^\ast$-measurable. Here, $\mathcal{I}_{t}^{\ast}$ is the smallest $\sigma$-algebra containing (1) $\mathcal{I}_{t}$ and (2) $(Z_1^{[0,t]}(s))_{s\geq 0}$, that is it contains the type-0 information up to time $t$ as well as all information from the early clones. We obtain that for each $n$, $X_{\tau_n'-T_n}^{T_n,(i)}|\mathcal{I}_{T_n}\overset{d}{=}X_{s}|s=\tau_n'-T_n$ on $\{i\leq Z_0(T_n)\}$ and $(X_{\tau_n'-T_n}^{T_n,(i)})_{t\geq T_n}|\mathcal{I}_T$ are i.i.d. over $i\in \{1,\dots,Z_0(T_n)\}$. Then, the same argument used in Theorem \ref{thm:clt} can be recycled to show that the scaled deviation evaluated at $\tau_n'$ obeys a central limit theorem. So, the reason for approximating $\tau_n$ by $\tau_n'$ is that the latter time is measurable with respect to the information of the early clones whilst the former time is not.

The most involved part is to prove that $X_{\tau_n',j}$ is a good approximation for $X_{\tau_n,j}$. To do so, we first show that $\tau_n'$ is a good approximation of $\tau_n$ at large $n$. This is intuitively clear since the early clones form the bulk of the population. Then, we decompose $|X_{\tau_n',j}-X_{\tau_n,j}|$ into various martingale and monotone pieces, showing each tends to $0$ in probability.

The prefactor $n^{-\lambda_0/(2\lambda_1)}$ and the extra source of randomness $W^{-\lambda_0/(2\lambda_1)}$ in Theorem \ref{thm:clt_detection_size} arise from the approximations $\tau_n\approx \frac{1}{\lambda_1}\log(n/W)$ and (thus):
\[e^{-\lambda_0\tau_n/2}\approx W^{\lambda_0/(2\lambda_1)}n^{-\lambda_0/(2\lambda_1)}.\]
Therefore, $\overline{X}_{n,j}\approx W^{-\lambda_0/(2\lambda_1)}X_{\tau_n,j}$ and Theorem \ref{thm:clt_detection_size} is shown.

\subsection{Deviations of the estimator $\widehat{b}(t)$}\label{sec:estimator_ci}

In our previous work, we showed that on $\Omega_0^\infty$:
\begin{align*}
    \frac{S_j(t)}{S_{\geq 1}(t)}\to F_j(b)
\end{align*}
almost surely. Here:
\begin{align*}
    F_j(\ell):=\frac{\frac{\lambda_0}{a+\ell-d}\frac{\Gamma (j)\Gamma(1+\lambda_0/(\ell+a-d))}{\Gamma(\lambda_0/(\ell+a-d)+1+j)}{}_2F_1(\lambda_0/(\ell+a-d),j;\lambda_0/(\ell+a-d)+1+j;d/(a+\ell))}{{}_2F_1(1,\lambda_0/(\ell+a-d);\lambda_0/(\ell+a-d)+1;d/(a+\ell))}
\end{align*}

In our previous work, we showed that $F_1(\ell)$ is invertible over $\ell\in [0,\infty)$ enabling the construction of the estimator $\widehat{b}(t)=F_1^{-1}(f_1(t))$ where $f_1(t)=S_1(t)/S_{\geq 1}(t)$. We argued that $\widehat{b}(t)\to b$ as $t\to\infty$ and $\widehat{b}(\tau_n)\to b$ as $n\to\infty$ almost surely on $\Omega_0^\infty$ by using our first-order asymptotic convergence results for $S_1(t),S_{\geq 1}(t)$. This consistency is plotted in Figure \ref{fig:b_1_consistency}A. Here, we use our second-order asymptotics to provide a central limit theorem for $\widehat{b}(t)$ and $\widehat{b}(\tau_n)$. In this section, we have the following notation:
\begin{align*}
        &\mu_j=(av+\mu)\int e^{-\lambda_0 s}P_j(s)\,ds=\lim_{t\to\infty}\mathbb{E}[e^{-\lambda_0 t}S_j(t)]\\
        &\mu_{\geq j}=\sum_{k\geq j}\mu_k\\
        &f_{\leq j}(t)=S_{\leq j}/S_{\geq 1}.
    \end{align*}

Our first result states:
\begin{restatable}{proposition}{ConfidenceInterval}\label{prop:confidence_interval}
    \begin{align*}
        e^{\lambda_0 t/2}(\widehat{b}(t)-b)|\Omega_0^\infty\overset{t\to\infty}{\Rightarrow} \mathcal{N}
    \end{align*}
    $\mathcal{I}_\infty$-stably. $\mathcal{N}$ is a mixture of a 0-mean Gaussian with a random variance depending on $Y$. In particular:
    \[\mathbb{V}[\mathcal{N}|Y]=\frac{1}{(F'(b)\mu_{\geq 1})^2Y}\left(\sigma_{1,1}+\frac{\mu_1^2}{\mu_{\geq 1}^2}\sum_{i=1,j=1}^\infty\sigma_{i,j}-2\frac{\mu_1}{\mu_{\geq 1}}\sum_{j=1}^\infty \sigma_{1,j}\right).\]
    Furthermore:
    \begin{align*}
        n^{\lambda_0/(2\lambda_1)}(\widehat{b}(\tau_n)-b)|\Omega_0^\infty\overset{n\to\infty}{\Longrightarrow}\mathcal{N}',
    \end{align*}
    $\mathcal{I}_\infty$-stably. $\mathcal{N}'$ is a mixture of a 0-mean Gaussian with a random variance depending on $(Y,W)$. In particular:
    \begin{align*}
        \mathbb{V}[\mathcal{N}'|(Y,W)]=W^{\lambda_0/\lambda_1}\mathbb{V}[\mathcal{N}|Y].
    \end{align*}
\end{restatable}
\begin{proof}
    Setting $g(x,y)=x/y$, we may write $f_1(t)=g(e^{-\lambda_0 t}S_1(t),e^{-\lambda _0 t}S_{\geq 1}(t))$. Let $\mu_1=\lim_{t\to\infty}e^{-\lambda_0 t}\mathbb{E}[S_1(t)]$ and $\mu_{\geq 1}=\lim_{t\to\infty}e^{-\lambda_0 t}\mathbb{E}[S_{\geq 1}(t)]$. By the multivariate Delta method:
\[e^{\lambda_0 t/2}(f_1(t)-\frac{\mu_1}{\mu_{\geq 1}})|\Omega_0^\infty\overset{t\to\infty}{\Rightarrow} \nabla g(Y\mu_1,Y\mu_{\geq 1})\begin{pmatrix}
X_{\infty,1}\\
X_{\infty,\geq 1}
\end{pmatrix}.\]
Remarking that $b=F^{-1}(\mu_1/\mu_{\geq 1})$ and $\widehat{b}(t)=F^{-1}(f_1(t))$ and applying the Delta method once more:
\begin{align*}
    e^{\lambda_0 t/2}(\widehat{b}(t)-b)|\Omega_0^\infty=e^{\lambda_0 t/2}(F^{-1}(f_1(t))-F^{-1}(\frac{\mu_1}{\mu_{\geq 1}}))|\Omega_0^\infty\overset{t\to\infty}{\Rightarrow} \frac{1}{F'(b)}\nabla g(Y\mu_1,Y\mu_{\geq 1})\begin{pmatrix}
        X_{\infty,1}\\
        X_{\infty,\geq 1}
    \end{pmatrix}.
\end{align*}
\end{proof}

Sample scaled deviations from simulation are shown in Figure \ref{fig:b_1_consistency}B, and the limiting distribution $\mathcal{N}'$ is compared to the empirical distribution in Figure \ref{fig:b_1_consistency}C.

In fact, we can construct a class of estimators using information of the fraction of clones of small size $\leq j$. Specifically, defining $F_{\leq j}(\ell):=\sum_{k=1}^j F_k(\ell)$, we have that $S_{\leq j}(t)/S_{\geq 1}(t)\to F_{\leq j}(b)$, with $\ell \to F_{\leq j}(\ell)$ still being continuous and monotone decreasing. Therefore, $\widehat{b}_{\leq j}(\cdot):=F_{\leq j}^{-1}(f_{\leq j}(\cdot))$ is still an asymptotically consistent estimator of $b$. Then, Proposition \ref{prop:confidence_interval} can easily be extended to show that:
\begin{restatable}{proposition}{ConfidenceIntervalExtension}\label{prop:confidence_interval_extension}
    \begin{align*}
        &e^{\lambda_0 t/2}(\widehat{b}_{\leq j}(t)-b)|\Omega_0^\infty\overset{t\to\infty}{\Longrightarrow} \mathcal{N}_{\leq j},\\
        &n^{\lambda_0/(2\lambda_1)}(\widehat{b}_{\leq j}(\tau_n)-b)|\Omega_0^\infty\overset{t\to\infty}{\Longrightarrow} \mathcal{N}_{\leq j}'.
    \end{align*}

    $\mathcal{I}_\infty$-stably. Here, $\mathcal{N}_{\leq j},\mathcal{N}_{\leq j}'$ are mixtures of a $0$-mean Gaussian with a random variance measurable with respect to $Y$ and $(Y,W)$, respectively. Specifically:
    \begin{align*}
        &\mathbb{V}[\mathcal{N}_{\leq j}|Y]=Y^{-1}\sigma_{b,\leq j}^2,\\
        &\mathbb{V}[\mathcal{N}_{\leq j}'|(Y,W)]=W^{\lambda_0/\lambda_1}Y^{-1}\sigma_{b,\leq j}^2.
    \end{align*}
    Here:
    \begin{align*}
        \sigma_{b,\leq j}^2:=\frac{1}{(F_{\leq j}'(b)\mu_{\geq 1})^2}\left(\sum_{i_1,i_2=1}^j \sigma_{i_1,i_2}-2\frac{\mu_1+\cdots+\mu_j}{\mu_{\geq 1}}\sum_{i=1}^j \sum_{k=1}^\infty\sigma_{i,k}+\frac{(\mu_1+\cdots+\mu_j)^2}{\mu_{\geq 1}^2}\sum_{k_1,k_2=1}^{\infty}\sigma_{k_1,k_2}\right)
    \end{align*}

\end{restatable}

\begin{figure}
    \centering
    \includegraphics[width=\linewidth]{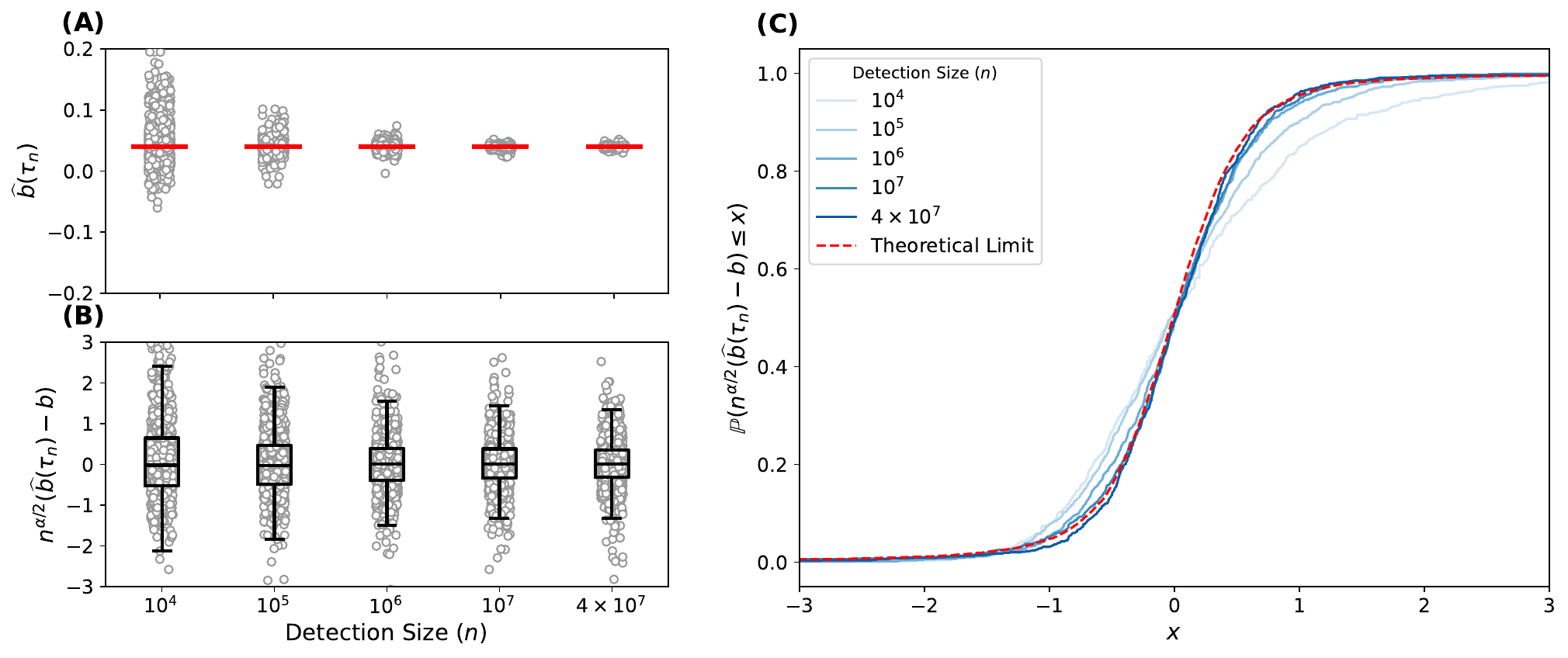}
    \caption{First and second-order asymptotics for the estimator $\widehat{b}$ at large detection sizes. \textbf{(A)} Red lines show the true fitness enhancement $b$. Each grey dot represents the estimator $\widehat{b}(\tau_n)$ computed from simulation. \textbf{(B)} Scaled deviations of $\widehat{b}(\tau_n)$ are shown. Boxplots summarize the first, second, and third quantile of the empirical distribution, with whiskers encircling the points lying within 1.5 IQRs of the first and third quantiles. \textbf{(C)} Empirical CDF of the scaled deviation of $\widehat{b}(\tau_n)$ are shown and compared to the distribution $\mathcal{N}^{'}$.}
    \label{fig:b_1_consistency}
\end{figure}

In fact $\sigma_{b,\leq j}^2$ asymptotically increases to $\infty$ as $j\to\infty$ at a known rate:

\begin{proposition}\label{prop:asymp_var_estimators}
    \[\sigma_{b,\leq j}^2=\Theta_{j\to\infty}(j^{\lambda_0/\lambda_1}(\log j)^{-2}).\]
\end{proposition}
\begin{proof}
    See Section \ref{app:asymp_var_estimators}.
\end{proof}

This implies that there exists a solution to the minimization problem $j^\ast\in \arg \min_{j\in\mathbb{N}} \sigma_{b,\leq j}^2$. $\widehat{b}_{\leq j^\ast}$ is the best estimator in the class $\{\widehat{b}_{\leq j}\}_{j\in\mathbb{N}}$ in terms of minimizing the asymptotic variance. At default parameters, $j^\ast=10$ (Figure \ref{fig:b_j_deviation}A). Empirical 75th and 25th quantiles behave as expected across $b_{\leq j}(\tau_n)$ fixing $n=4\cdot 10^7$ and varying $j$ (Figure \ref{fig:b_j_deviation}B).

These variance results can be used to construct asymptotic $(1-\alpha)$-confidence intervals. Writing:
\begin{align*}
    V_j(b):=\mu_{\geq 1}\sigma_{b,\leq j}^2
\end{align*}
and using $\lim_{t\to\infty}\exp(-\lambda_0 t)S_{\geq 1}(t)=Y\mu_{\geq1},\lim_{n\to\infty}n^{-\lambda_0/\lambda_1}S_{\geq 1}(\tau_n)=YW^{-\lambda_0/\lambda_1}\mu_{\geq 1}$ on $\Omega_0^\infty$:
\begin{align*}
    \widehat{b}_{\leq j}(t)\pm z_{1-\alpha/2}\sqrt{\frac{V_j(\widehat{b}_{\leq j}(t))}{S_{\geq 1}(t)}}\text{ and }\widehat{b}_{\leq j}(\tau_n)\pm z_{1-\alpha/2}\sqrt{\frac{V_j(\widehat{b}_{\leq j}(\tau_n))}{S_{\geq 1}(\tau_n)}},
\end{align*}
are asymptotic $(1-\alpha)$-confidence intervals for $b$. $z_p$ denotes the $p$th quantile of a standardized normal.

\begin{figure}
    \centering
    \includegraphics[width=\linewidth]{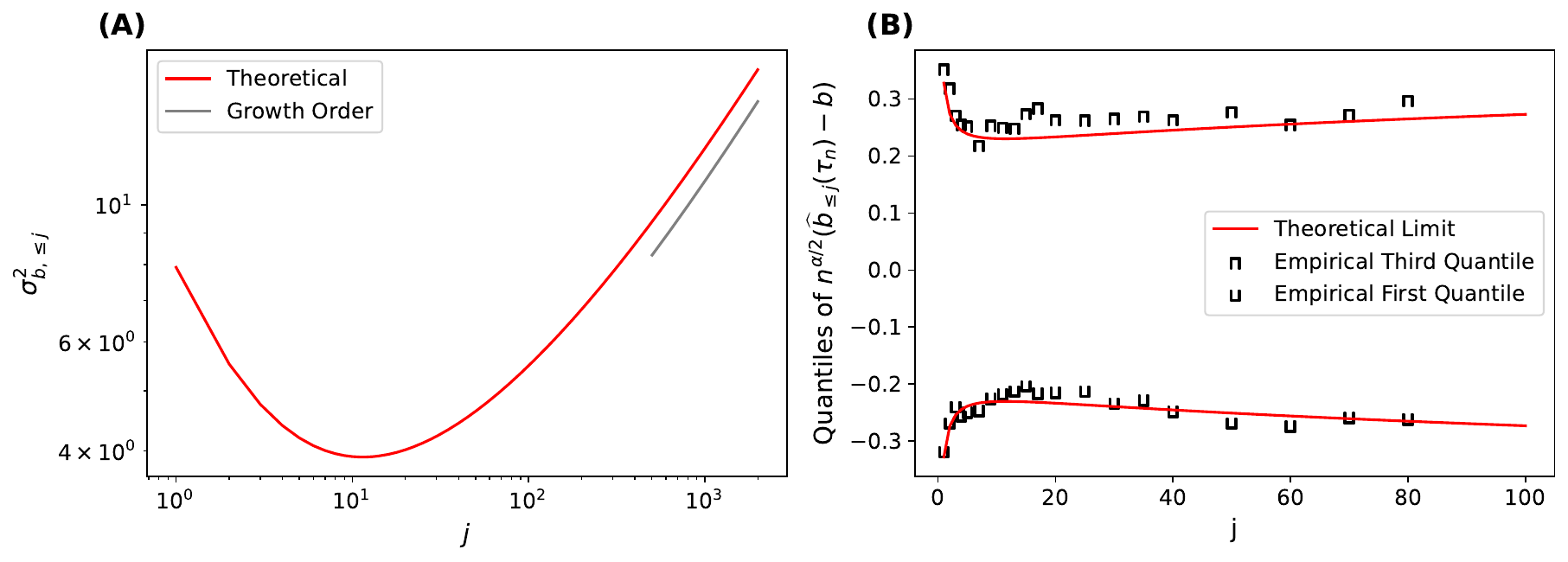}
    \caption{Limiting deviations of the class of estimators $\{\widehat{b}_{\leq j}(\tau_n)\}_{j\geq 1}$. \textbf{(A)} The deterministic part of the variance $\sigma_{b,\leq j}^2$ is shown and compared across $1\leq j\leq 2\cdot 10^3$. The growth order $\sigma_{b,\leq j}^2\asymp j^{\lambda_0/\lambda_1}(\log j)^{-2}$ is shown. \textbf{(B)} First and third quantiles are compared between simulation of the scaled deviation (for $n=4\cdot 10^7$) and the distribution $\mathcal{N}_{\leq j}'$.}
    \label{fig:b_j_deviation}
\end{figure}

\section{Discussion}

Our results describe the large time and size behavior of the SFS and apply them to obtain second-order asymptotics for estimators of the fitness enhancement. Up to constants $C_1,C_2$, we can (heuristically) write that $S_j(t)$ and $S_j(\tau_n)$ on $\Omega_0^\infty$ behave as:
\begin{align*}
    &S_j(t)\approx e^{\lambda_0 t}\frac{C_1}{j^{1+\lambda_0/\lambda_1}}+\mathcal{O}\left(e^{(\lambda_0/2)t}\frac{\sqrt{\log j}}{j^{1/2+\lambda_0/(2\lambda_1)}}\right),\\
    &S_j(\tau_n)\approx n^{\lambda_0/\lambda_1}\frac{C_2}{j^{1+\lambda_0/\lambda_1}}+\mathcal{O}\left(n^{\lambda_0/(2\lambda_1)}\frac{\sqrt{\log j}}{j^{1/2+\lambda_0/(2\lambda_1)}}\right)
\end{align*}
in the sense that $t$ (resp. $n$) is taken to be large first and then $j$.

The second-order large-time, large-$j$ behavior is due to the deviations in the Gaussian process being close to uncorrelated (Corollary \ref{cor:clt}). This comes from the fact that the variance in the deviations of a site $\sigma_{j,j}\approx j^{-\lambda_0/\lambda_1-1}$ is much more heavy-tailed (in $j$) compared to the covariance if the deviations of nearby sites, say sites $j$ and $j+1$, $\sigma_{j,j+1}\approx j^{-2-\lambda_0/\lambda_1}$.

Then, we exploited the central limit theorems on $\{S_{\geq 1}(t)\}\cup(S_j(t))_{j=1}^\infty,\{S_{\geq 1}(\tau_n)\}\cup(S_j(\tau_n))_{j=1}^\infty$ on $\Omega_0^\infty$ to obtain deviation results for estimators of the fitness increase $b$.  We can understand the location of $j^\ast$ to be contributed by two factors:  large-$j$ (1) includes data from more clones to estimate $b$ but (2) makes $F_{\leq j}(\ell)$ vary less for $\ell$ in neighborhoods of $b$ so that $F_{\leq j}^{-1}$ varies more around $F_{\leq j}(b)$. Indeed, this trade-off is seen in both subpanels of Figure \ref{fig:b_j_deviation}. Where these forces balance is where $\sigma_{b,\leq j}^2$ is minimized. We can estimate this optimal $j^\ast$ by, say, using the plug-in estimator $\widehat{j}^\ast\in \argmin_{j\in \mathbb{N}} \sigma_{\widehat{b}_{\leq j_{\text{prior}}},\leq j}^2$, where $j_{\text{prior}}\in\mathbb{N}$ is an initial guess.

\section{Proofs}\label{sec:proofs}

\subsection{Important Auxiliary Processes and Results}

For $t>0$ define the total number of mutations by time $t$ by $M(t)=\#\{i\geq 1:T_i\leq t\}$. We have:
\begin{lemma}
    $$
R_t:=M(t)-(av+\mu)\int_0^tZ_0(s)ds,
$$
is a $\mathcal{B}_t$ martingale (and thus a $\mathcal{B}_t^{\leq 1}$- and $\mathcal{I}_t$-martingale).
\end{lemma}
\begin{proof}
    See Lemma 9 of \cite{ahmed2026}.
\end{proof}

We have a straightforward representation of: $$W:=\lim_{t\to\infty} e^{-\lambda_1 t}Z(t)$$
\begin{lemma}\label{lem:W_form}
    \[W=\sum_{i}e^{-\lambda_1 T_i}Y_1^{(i)},\]
    where $Y_1^{(i)}=\lim_{t\to\infty} e^{-\lambda_1 t}Z_1^{(i)}(t)$ are i.i.d. deflated exponentials which are $0$ with probability $p_1$ and exponential with rate $q_1$ with probability $q_1$
\end{lemma}
\begin{proof}
    See Lemma 5 of \cite{ahmed2026}.
\end{proof}

\subsection{SFS central limits theorems at large detection times}\label{app:clt_time}

Here, we aim to prove Theorem \ref{thm:clt} and Corollary \ref{cor:clt}. First, we define how our covariance kernel should look.

\begin{lemma}\label{lem:sigma_large_time}
    \begin{align}
        \sigma_{j,k}&:=\lim_{t\to\infty}\Cov(X_{t,j},X_{t,k})\label{eq:sigmajk_def}\\
        &=\left(\frac{(av+\mu)^2(a(1-v)+d+\mu)}{\lambda_0}+\mu(av+\mu)\right)\int_0^\infty e^{-\lambda_0 s}\left(P_j(s)\int_0^s P_k(z)\,dz+P_k(s)\int_0^s P_j(z)\,dz\right)\,ds\nonumber\\
        &+\mathbf{1}_{k=j}(av+\mu)\int_0^\infty P_j(s)e^{-\lambda_0 s}\,ds\label{eq:sigma_jk_second_term}
    \end{align}
    
\end{lemma}
\begin{proof}
    Begin by supposing $k\neq j$. Notice, that by using the law of total covariance and conditioning on $\mathcal{F}_\infty$:
  \begin{align}
    &e^{-\lambda_0 t}\Cov[S_j(t)-(av+\mu)Y\int_0^t P_j(t-s)e^{\lambda_0 s}\,ds,S_k(t)-(av+\mu)Y\int_0^t P_k(t-s)e^{\lambda_0 s}\,ds]\label{eq:cov_computation}\\
    &=-e^{-\lambda_0 t}\mathbb{E}\left[\sum_{T_i\leq t}P_j(t-T_i)P_k(t-T_i)\right]\label{eq:cov_computation_1}\\
    &+e^{-\lambda_0 t}\Cov\left[\sum_{T_i\leq t}P_j(t-T_i)-(av+\mu)Y\int_0^t P_j(t-s)e^{\lambda_0 s}\,ds,\sum_{T_i\leq t}P_k(t-T_i)-(av+\mu)Y\int_0^t P_k(t-s)e^{\lambda_0 s}\,ds\right]\label{eq:cov_computation_2}
  \end{align}
  Using Campbell's theorem, the first term can be expressed as:
  \begin{align*}
    e^{-\lambda_0 t}\mathbb{E}\left[\sum_{T_i\leq t}P_j(t-T_i)P_k(t-T_i)\right]\overset{t\to\infty}{\to}(av+\mu)\int_0^\infty P_j(s)P_k(s)e^{-\lambda_0 s}\,ds.
  \end{align*}
  The second term requires more work to express. Write $L_t=e^{-\lambda_0 t}Z_0(t)$. It is a classical result that $L_t$ is a $\mathcal{B}_t$-martingale. Now, we will need to understand the variation and cross variation of $L_t$ and $R_t$. Letting $\Phi_b,\Phi_d$ denote the birth and death times of the type-0 population (birth times not including replication-dependent mutation events and death times including replication-free mutation events). Then:
  \begin{align}
    [L]_t&=\sum_{T\in (\Phi_b\cup \Phi_d)\cap [0,t]}e^{-2\lambda_0 T}\label{eq:L_var}\\
    [R]_t&=M_t\label{eq:R_var}\\
    [L,R]_t&=-\sum_{T\in \Phi\cap \Phi_d\cap [0,t]}e^{-\lambda_0 T}\label{eq:L_R_var}
  \end{align}
  
  By adding and subtracting $(av+\mu)\int_0^t P_j(t-s)Z_0(s)\,ds$, we can write:
  \begin{align}
    &e^{-\lambda_0 t}\Cov\left[\sum_{T_i\leq t}P_j(t-T_i)-(av+\mu)Y\int_0^t P_j(t-s)e^{\lambda_0 s}\,ds,\sum_{T_i\leq t}P_k(t-T_i)-(av+\mu)Y\int_0^t P_k(t-s)e^{\lambda_0 s}\,ds\right]\label{eq:term_total}\\
    &=e^{-\lambda_0 t}\Cov\left[\int_0^t P_j(t-s)\,dR_s,\int_0^t P_k(t-s)\,dR_s\right]\label{eq:term_Pj_R_Pk_R}\\
    &+e^{-\lambda_0 t}(av+\mu)^2\Cov\left[\int_0^t P_j(t-s)(Z_0(s)-Ye^{\lambda_0 s})\,ds,\int_0^t P_k(t-s)(Z_0(s)-Ye^{\lambda_0 s})\,ds\right]\label{eq:term_Pj_ds_Pk_ds}\\
    &+(av+\mu)e^{-\lambda_0 t}\Cov\left[\int_0^t P_j(t-s)\,dR_s,\int_0^t P_k(t-s)(Z_0(s)-Ye^{\lambda_0 s})\,ds\right]\label{eq:term_Pj_R_Pk_ds}\\
    &+(av+\mu)e^{-\lambda_0 t}\Cov\left[\int_0^t P_k(t-s)\,dR_s,\int_0^t P_j(t-s)(Z_0(s)-Ye^{\lambda_0 s})\,ds\right]\label{eq:term_Pk_R_Pj_ds}.
  \end{align}
  We now look at each term individually (although the last two terms are symmetric, so WLOG, we can simply look at term \ref{eq:term_Pj_R_Pk_ds}). By using Ito's isometry with \ref{eq:R_var}:
  \begin{align*}
      \ref{eq:term_Pj_R_Pk_R}&=e^{-\lambda_0 t}\mathbb{E}\left[\int_0^t P_j(t-s)\,dR_s\int_0^t P_k(t-s)\,dR_s\right]\\
      &=(av+\mu)e^{-\lambda_0 t}\int_0^t P_j(t-s)P_k(t-s)e^{\lambda_0 s}\,ds\\
      &\overset{t\to\infty}{\to}(av+\mu)\int_0^\infty P_j(s)P_k(s)e^{-\lambda_0 s}\,ds
  \end{align*}
  Now, rewrite $Z_0(s)-Ye^{\lambda_0s}=-e^{\lambda_0 s}\int_0^\infty\mathbf{1}_{z\geq s}\,dL_z$, so that:
  \begin{align*}
    \ref{eq:term_Pj_ds_Pk_ds}&=e^{-\lambda_0 t}(av+\mu)^2\int_0^t \int_{s_1}^t P_j(t-s_1)P_k(t-s_2)e^{\lambda_0 s_1}e^{\lambda_0 s_2}\mathbb{E}[\int_0^\infty \mathbf{1}_{z\geq s_1}\,dL_z\int_0^\infty \mathbf{1}_{z\geq s_2}\,dL_z]\,ds_2\,ds_1\\
    &+e^{-\lambda_0 t}(av+\mu)^2\int_0^t \int_{s_1}^t P_k(t-s_1)P_j(t-s_2)e^{\lambda_0 s_1}e^{\lambda_0 s_2}\mathbb{E}[\int_0^\infty \mathbf{1}_{z\geq s_1}\,dL_z\int_0^\infty \mathbf{1}_{z\geq s_2}\,dL_z]\,ds_2\,ds_1.
  \end{align*}
  Using a variant of Ito's formula for the product of two (different) stochastic integrals with martingale integrators, and if $s_2>s_1$ and \ref{eq:L_var}:
  \begin{align*}
\mathbb{E}[\int_0^\infty \mathbf{1}_{z\geq s_1}\,dL_z\int_0^\infty \mathbf{1}_{z\geq s_2}\,dL_z]= \mathbb{E}[\int_{s_2}^\infty\,d[L]_z]&=(a(1-v)+d+\mu)\int_{s_2}^\infty e^{-2\lambda_0 z}e^{\lambda_0 z}\,dz\\
&=\frac{(a(1-v)+d+\mu)}{\lambda_0}e^{-\lambda_0 s_2}.
  \end{align*}
  Performing the substitution $s_1\to t-s_1$ and $s_2\to t-s_2$, we can obtain the limit of \ref{eq:term_Pj_ds_Pk_ds} as:
  \begin{align*}
    \ref{eq:term_Pj_ds_Pk_ds}\overset{t\to\infty}{\to}&\frac{(av+\mu)^2(a(1-v)+d+\mu)}{\lambda_0}\int_0^\infty \int_0^{s_1}P_j(s_1)P_k(s_2)e^{-\lambda_0 s_1}\,ds_2\,ds_1\\
    &+\frac{(av+\mu)^2(a(1-v)+d+\mu)}{\lambda_0}\int_0^\infty \int_0^{s_1}P_k(s_1)P_j(s_2)e^{-\lambda_0 s_1}\,ds_2\,ds_1
  \end{align*}
  We now calculate the cross terms:
  \begin{align*}
    \ref{eq:term_Pj_R_Pk_ds}&=-e^{-\lambda_0 t}(av+\mu)\int_0^t P_k(t-s)e^{\lambda_0 s}\mathbb{E}[\int_0^\infty \mathbf{1}_{z\leq t}P_j(t-z)\,dR_z\int_0^\infty\mathbf{1}_{s\leq z}\,dL_z]\,ds.
  \end{align*}
  Using \ref{eq:L_R_var}:
  \begin{align*}
    \mathbb{E}[\int_0^\infty \mathbf{1}_{z\leq t}P_j(t-z)\,dR_z\int_0^\infty\mathbf{1}_{s\leq z}\,dL_z]=-\mu\int_s^t P_j(t-z)\,dz,
  \end{align*}
  therefore, upon a change of variables:
  \begin{align*}
    \ref{eq:term_Pj_R_Pk_ds}\overset{t\to\infty}{\to}\mu(av+\mu)\int_0^\infty P_k(s)e^{-\lambda_0 s}\int_0^s P_j(z)\,dz\,ds.
  \end{align*}
  Therefore, similarly:
  \begin{align*}
  \ref{eq:term_Pk_R_Pj_ds}\overset{t\to\infty}{\to}\mu(av+\mu)\int_0^\infty P_j(s)e^{-\lambda_0 s}\int_0^s P_k(z)\,dz\,ds.
  \end{align*}
  
  Combining all these terms, the result follows in the case $k\ne j$. In the case $j=k$, only term \ref{eq:cov_computation_1} is affected. In particular it should instead be:
  \begin{align}
      e^{-\lambda_0 t}\mathbb{E}\left[\sum_{T_i\leq t}P_j(t-T_i)(1-P_j(t-T_i))\right]\to &(av+\mu)\int_0^\infty P_j(s)(1-P_j(s))e^{-\lambda_0 s}\,ds\nonumber\\
      &=-(av+\mu)\int_0^\infty P_j(s)^2e^{-\lambda_0 s}\,ds\nonumber\\
      &+(av+\mu)\int_0^\infty P_j(s)e^{-\lambda_0 s}\,ds,\label{eq:cov_computation_extra_term}
  \end{align}
  which has \ref{eq:cov_computation_extra_term} as an extra term compared to the $k\neq j$ case.
\end{proof}

\begin{corollary}\label{cor:sigmajk_large_jk}
    If $j\neq k$:
    \begin{align*}
        \sigma_{j,k}=\frac{q_1}{\lambda_1}\left(\frac{(av+\mu)^2(a(1-v)+d+\mu)}{\lambda_0}+\mu(av+\mu)\right)\frac{1}{jk}\left(\frac{q_1^{1-\lambda_0/\lambda_1}\Gamma(\lambda_0/\lambda_1+1)}{\lambda_1 (j+k)^{\lambda_0/\lambda_1}}+\mathcal{O}_{j+k\to\infty}(1/(j+k)^{\lambda_0/\lambda_1+1})\right)
    \end{align*}
    If $j=k$:
    \begin{align*}
        \sigma_{j,j}=(av+\mu)\frac{q_1^{1-\lambda_0/\lambda_1}\Gamma(\lambda_0/\lambda_1+1)}{\lambda_1j^{\lambda_0/\lambda_1+1}}+\mathcal{O}_{j\to\infty}(\frac{1}{j^{\lambda_0/\lambda_1+2}})
    \end{align*}
\end{corollary}
\begin{proof}
    Begin with $j\neq k$. Firstly, let:
    \begin{align*}
        \alpha(z)=\frac{de^{\lambda_1 z}-d}{(a+b)e^{\lambda_1 z}-d},\beta(z)=\frac{(a+b)e^{\lambda_1 z}-(a+b)}{(a+b)e^{\lambda_1 z}-d}.
    \end{align*}
    Thus:
    \[P_j(z)=(1-\alpha(z))(1-\beta(z))\beta(z)^{j-1}.\]
    by Section 3, Eq. 8 of \cite{durrettbranching2015}.
    
    Now remark:
    \begin{align*}
        d\beta&=\lambda_1\left(\frac{(a+b)e^{\lambda_1 z}}{(a+b)e^{\lambda_1 z}-d}-\frac{(a+b)^2e^{2\lambda_1 z}-(a+b)^2e^{\lambda_1 z}}{((a+b)e^{\lambda_1 z}-d)^2}\right)\,dz\\
        &=\frac{\lambda_1}{q_1}\left(\frac{(a+b-d)^2e^{\lambda_1 z}}{((a+b)e^{\lambda_1 z}-d)^2}\right)\,dz\\
        &=\frac{\lambda_1}{q_1}P_1(z)\,dz.
    \end{align*}
    Thus, using the change of variables $u=\beta(z)$:
    \[\int_0^s P_k(z)\,dz=\frac{q_1}{\lambda_1}\int_0^{\beta(s)}u^{k-1}\,du=\frac{q_1}{\lambda_1 k}\beta(s)^k.\]
    Therefore:
    \begin{align}
        \int_0^\infty e^{-\lambda_0 s}P_j(s)\int_0^s P_k(z)\,dz&=\frac{q_1}{\lambda_1 k}\int_0^\infty e^{-\lambda_0 s}P_{j+k}(s)\,ds\nonumber\\
        &=\frac{q_1}{\lambda_1 k}\left(\frac{q_1^{1-\lambda_0/\lambda_1}\Gamma(\lambda_0/\lambda_1+1)}{\lambda_1 (j+k)^{\lambda_0/\lambda_1+1}}+\mathcal{O}(1/(j+k)^{\lambda_0/\lambda_1+2})\right).\label{eq:part_sigmajk_asymp}
    \end{align}
    The last equality comes from Proposition 1 of \cite{ahmed2026}. Using \ref{eq:sigmajk_def}, the above \ref{eq:part_sigmajk_asymp}, and the above swapping $j$ and $k$ concludes the result for this case.

    In the case $j=k$, the second term \ref{eq:sigma_jk_second_term} dominates \ref{eq:sigmajk_def} for large $j$.
\end{proof}
\begin{remark}\label{rem:sigma_large_time_unif}
    A consequence of Lemma \ref{lem:sigma_large_time} and Corollary \ref{cor:sigmajk_large_jk} is that
    \begin{align}\label{eqn:2ndMomentX}
    \sup_{t<\infty}E[X_{t,j}^2]=\mathcal{O}_{j\to\infty}(1/j^{\lambda_0/\lambda_1+1}).
    \end{align}
    This is because as per Lemma \ref{lem:sigma_large_time}, we may decompose $\sigma_{j,j}$ into terms \ref{eq:term_Pj_R_Pk_R}, \ref{eq:term_Pj_ds_Pk_ds}, \ref{eq:term_Pj_R_Pk_ds}, \ref{eq:term_Pk_R_Pj_ds}, and \ref{eq:cov_computation_extra_term}, each of whose magnitudes are increasing and whose limits are $\mathcal{O}_{j\to\infty}(1/j^{\lambda_0/\lambda_1+1})$
\end{remark}
To establish Theorem \ref{thm:clt}, we first prove weak convergence with respect to the product topology, and then prove weak convergence with respect to the $\ell^p$ topology by showing that a tightness condition is satisfied. First, we show uniform fourth-moment control over $X_{t,j}$, which will be needed for the finite-dimensional convergence argument.

\begin{lemma}\label{lem:control_fourth_moment}
    \[\sup_{t\geq 0}\mathbb{E}[|X_{t,j}|^4]<\infty.\]
\end{lemma}
\begin{proof}
    See Appendix Section \ref{sec:control_fourth_moment}
\end{proof}

\begin{proposition}\label{prop:weak_topology_convergence}
    \begin{align}
        (X_{t,j})_{j=1}^\infty|\Omega_0^\infty\overset{t\to\infty}{\Rightarrow}(X_{\infty,j})_{j=1}^\infty|\Omega_0^\infty
    \end{align}
    $\mathcal{I}_\infty$-stably, under the product topology (or, equivalently, in finite dimensions).
\end{proposition}

\begin{proof}
We provide the argument to show weak convergence and then in a remark show where the proof is adapted to show $\mathcal{I}_\infty$-stable convergence. Fix $T>0$ large. For $t>T$, we can write $Z_0(t)=\sum_{i=1}^{Z_0(T)}Z_0^{T,(i)}(t-T)$, where $Z_0^{T,(i)}(t-T)$ denotes the descendants present at time $t$ of the $i$th cell present at time $T$. Note that $Z_0^{T,(i)}$, conditional on $\mathcal{I}_T$, are i.i.d. copies of the type-0 process for $1\leq i\leq Z_0(T)$. For $t\geq T$, define $S_j^{[0,T]}(t):=\sum_{T_i\leq T}\mathbf{1}_{Z_1^{(i)}(t-T_i)=j}$ denote the contribution to the SFS from mutants produced in the interval $[0,T]$. Let $S_j^{T,(i)}(t-T)$ be the SFS of the lineage associated with $Z_0^{T,(i)}$. Note then $S_j(t)=S_j^{[0,T]}(t)+\sum_{i=1}^{Z_0(T)}S_j^{T,(i)}(t-T)$. Furthermore, remark that $Y=e^{-\lambda T}\sum_{i=1}^{Z(T)}Y^{T,(i)}$, where $Y^{T,(i)}=\lim_{t\to\infty}e^{-\lambda_0 (t-T)}Z_0^{T,(i)}(t-T)$. Therefore:
\begin{align}\label{eq:Xt_Rep}
    X_{t,j}=e^{-\lambda_0 T/2}\sum_{i=1}^{Z_0(T)}X^{T,(i)}_{t-T,j}+e^{-\lambda_0 t/2}S_j^{[0,T]}(t)-(av+\mu)e^{-\lambda_0 t/2}Y\int_0^T P_j(t-s)e^{\lambda_0 s}\,ds,
\end{align}
where:
\begin{align*}
    X^{t_2,(i)}_{t_1,j}:=e^{-\lambda_0 t_1/2}\left(S_j^{t_2,(i)}(t_1)-(av+\mu)Y^{t_2,(i)}\int_0^{t_1}P_j(t_2-s)e^{\lambda_0 s}\,ds\right).
\end{align*}
The last two terms in the above representation of $X_{t,j}$ go to $0$ almost surely so that to investigate the large time behavior of $X_{t,j}$, it suffices to look at the first term. To show convergence in the product topology, it suffices to show convergence of characteristic functions of finite-dimensional projections of $(X_{t,j})_{j=1}^\infty|\Omega_0^\infty$. Fix any $n$ and sequence $j_1<j_2<\cdots<j_n$. Let $\Sigma$ be a $n\times n$ matrix with the $k_1,k_2$ entry being $\sigma_{j_{k_1},j_{k_2}}$. Let $\boldsymbol{\theta}\in\mathbb{R}^n$, $\mathbf{X}_{t-T}^{T,(i)}:=(X_{t-T,j_k}^{T,(i)})_{k=1}^n$, and $\Sigma_{t,T}$ denote the covariance matrix of $\mathbf{X}_{t-T}^{T,(i)}$. Noting $\mathbb{E}[X_{t-T,j}^{T,(i)}]=0$ for all $j$, we see:
\begin{align}
    \mathbb{E}\left[\exp\left[ie^{-\lambda_0 T/2}\sum_{i=1}^{Z_0(T)}\langle \boldsymbol{\theta},\boldsymbol{X}_{t-T}^{T,(i)}\rangle\right]|\mathcal{I}_T\right]&=\mathbb{E}\left[\exp\left[i e^{-\lambda_0 T/2} \langle \boldsymbol{\theta}, \boldsymbol{X}_{t-T}^{T,(1)}\rangle\right]\right]^{Z_0(T)}\nonumber\\
    &=\left(1-\frac{1}{2}e^{-\lambda_0 T}\boldsymbol{\theta}^\intercal\Sigma_{t,T}\boldsymbol{\theta}+o_{T\to\infty}(e^{-\lambda_0 T})\right)^{Z_0(T)}\label{eq:conditional_T_rewrite}
\end{align}
This $o_{T\to\infty}$ term is uniform in $t$. To see this, recall that by a Taylor error bound:

\begin{align*}
    &\left|\mathbb{E}\left[\exp\left[i e^{-\lambda_0 T/2} \langle \boldsymbol{\theta}, \boldsymbol{X}_{t-T}^{T,(1)}\rangle\right]\right]-1+\frac{1}{2}e^{-\lambda_0 T}\boldsymbol{\theta}^\intercal \Sigma_{t,T}\boldsymbol{\theta}\right|\\
    &\leq \frac{1}{6}e^{-3\lambda_0T/2}\mathbb{E}[|\langle \boldsymbol{\theta},\mathbf{X}_{t-T}^{T,(i)}\rangle|^3].
\end{align*}
Note that $\mathbf{X}_{t-T}^{T,(i)}\sim \mathbf{X}_{s}^{0,(1)}$, where $s=t-T$. The last term above is a finite linear combination of terms of the form $\mathbb{E}[X_{s,j_{k_{1}}}X_{s,j_{k_{2}}}X_{s,j_{k_{3}}}]$ where possibly certain $k_1,k_2,k_3$ may be the same. By Holder's inequality, this is controlled via $\mathbb{E}[X_{s,j_{k_1}}^4]^{1/4}\mathbb{E}[X_{s,j_{k_2}}^4]^{1/4}\mathbb{E}[X_{s,j_{k_3}}^4]^{1/4}$. Subsequently, each of these factors are uniformly bounded by Lemma \ref{lem:control_fourth_moment}. Therefore:
\begin{align*}
    e^{-3\lambda_0T/2}\mathbb{E}[|\langle \boldsymbol{\theta},\mathbf{X}_{t-T}^{T,(i)}\rangle|^3]\leq Ke^{-3\lambda_0 T/2},
\end{align*}
where $K$ is independent of $t,T$.

Recall that for two complex numbers $z,w$ of modulus at most $1$ and a natural number $n\geq 0$:
\[|z^n-w^n|\leq n|z-w|.\]
Thus, combining this with \ref{eq:conditional_T_rewrite}
\begin{align*}
    &\limsup_{t\to\infty}\left|\mathbb{E}[e^{ie^{-\lambda_0 T/2}\sum_{i=1}^{Z_0(T)}\langle \boldsymbol{\theta},\boldsymbol{X}_{t-T}^{T,(i)}\rangle}|\mathcal{I}_T]-e^{-\frac{1}{2}e^{-\lambda_0 T}Z_0(T)\boldsymbol{\theta}^\intercal\Sigma\boldsymbol{\theta}}\right|\\
    &\leq Z_0(T)\limsup_{t\to\infty}\left|1-\frac{1}{2}e^{-\lambda_0 T}\boldsymbol{\theta}^\intercal \Sigma_{t,T}\boldsymbol{\theta}-\exp\left[-\frac{1}{2}e^{-\lambda_0 T}\boldsymbol{\theta}^\intercal\Sigma\boldsymbol{\theta}\right]\right|+Z_0(T)o_{T\to\infty}(e^{-\lambda_0 T})\\
    &\leq Z_0(T)\frac{1}{2}e^{-\lambda_0 T}\limsup_{t\to\infty}\left|\boldsymbol{\theta}^\intercal(\Sigma_{t,T}-\Sigma)\boldsymbol{\theta}+\mathcal{O}_{T\to\infty}\left(e^{-\lambda_0 T}\right)\right|+Z_0(T)o_{T\to\infty}(e^{-\lambda_0 T}).
\end{align*}
But, $\Sigma_{t,T}\to \Sigma$ as $t\to\infty$, so that:
\begin{align*}
    \limsup_{t\to\infty}\left|\mathbb{E}[e^{ie^{-\lambda_0 T/2}\sum_{i=1}^{Z_0(T)}\langle \boldsymbol{\theta},\boldsymbol{X}_{t-T}^{T,(i)}\rangle}|\mathcal{I}_T]-e^{-\frac{1}{2}e^{-\lambda_0 T}Z_0(T)\boldsymbol{\theta}^\intercal\Sigma\boldsymbol{\theta}}\right|=Z_0(T)o_{T\to\infty}(e^{-\lambda_0 T})
\end{align*}

Letting $\mathbf{X}_t=(X_{t,j_k})_{k=1}^n$, we can thus write:
\begin{align}
    &\limsup_{t\to\infty}\left|\mathbb{E}[\left(e^{i\langle \boldsymbol{\theta},\boldsymbol{X}_t\rangle}-e^{-\frac{1}{2}Y\boldsymbol{\theta}^\intercal\Sigma\boldsymbol{\theta}}\right)\mathbf{1}_{Y>0}]\right|\nonumber\\
    &\leq \limsup_{T\to\infty}\limsup_{t\to\infty}\left|\mathbb{E}[\left(e^{i\langle \boldsymbol{\theta},\boldsymbol{X}_t\rangle}-e^{ie^{-\lambda_0 T/2}\sum_{i=1}^{Z_0(T)}\langle \boldsymbol{\theta},\boldsymbol{X}_{t-T}^{T,(i)}\rangle}\right)\mathbf{1}_{Y>0}]\right|\label{eq:time_diff_initial}\\
    &+\limsup_{T\to\infty}\limsup_{t\to\infty}\left|\mathbb{E}[\left(e^{ie^{-\lambda_0 T/2}\sum_{i=1}^{Z_0(T)}\langle \boldsymbol{\theta},\boldsymbol{X}_{t-T}^{T,(i)}\rangle}-e^{-\frac{1}{2}e^{-\lambda_0 T}Z_0(T)\boldsymbol{\theta}^\intercal\Sigma\boldsymbol{\theta}}\right)\mathbf{1}_{Y>0}]\right|\label{eq:time_diff_semideterministic}\\
    &+\limsup_{T\to\infty}\left|\mathbb{E}[\left(e^{-\frac{1}{2}e^{-\lambda_0 T}Z_0(T)\boldsymbol{\theta}^\intercal\Sigma\boldsymbol{\theta}}-e^{-\frac{1}{2}Y\boldsymbol{\theta}^\intercal\Sigma\boldsymbol{\theta}}\right)\mathbf{1}_{Y>0}]\right|.\label{eq:time_diff_final}
\end{align}
From \eqref{eq:Xt_Rep} we can see that \ref{eq:time_diff_initial} converges to $0$ via dominated convergence, we can also apply dominated convergence to see that \ref{eq:time_diff_final} converges to $0$. For term \ref{eq:time_diff_semideterministic}, we can write $\mathbf{1}_{Y>0}=\mathbf{1}_{Y>0}-\mathbf{1}_{Z_0(T)>0}+\mathbf{1}_{Z_0(T)>0}$ and observe that $\mathbb{E}[|\mathbf{1}_{Y>0}-\mathbf{1}_{Z_0(T)>0}|]=\mathbb{P}(Z_0(T)>0,Y=0)=o_{T\to\infty}(1)$ so:
\begin{align*}
    \ref{eq:time_diff_semideterministic}&\leq \limsup_{T\to\infty}\limsup_{t\to\infty}\left|\mathbb{E}[\left(e^{ie^{-\lambda_0 T/2}\sum_{i=1}^{Z_0(T)}\langle \boldsymbol{\theta},\boldsymbol{X}_{t-T}^{T,(i)}\rangle}-e^{-\frac{1}{2}e^{-\lambda_0 T}Z_0(T)\boldsymbol{\theta}^\intercal\Sigma\boldsymbol{\theta}}\right)\mathbf{1}_{Z_0(T)>0}]\right|\\
    &=\limsup_{T\to\infty}\limsup_{t\to\infty}\left|\mathbb{E}[\left(\mathbb{E}\left[e^{ie^{-\lambda_0 T/2}\sum_{i=1}^{Z_0(T)}\langle \boldsymbol{\theta},\boldsymbol{X}_{t-T}^{T,(i)}\rangle}|\mathcal{I}_T\right]-e^{-\frac{1}{2}e^{-\lambda_0 T}Z_0(T)\boldsymbol{\theta}^\intercal\Sigma\boldsymbol{\theta}}\right)\mathbf{1}_{Z_0(T)>0}]\right|\\
    &=0.
\end{align*}
\end{proof}
\begin{remark}
    Now, to show stable convergence with respect to $\mathcal{I}_\infty$, it suffices to show stable convergence with respect to $\mathcal{I}_t$ for all $t\geq 0$. To see this, note by assumption:
    \begin{align*}
        \bigcup_{t\geq 0}\mathcal{I}_t&\subseteq \{A\in \mathcal{I}_\infty:\lim_{t\to\infty}\mathbb{E}[\mathbf{1}_Ag((X_{t,j})_{j=1}^\infty)]=\mathbb{E}[\mathbf{1}_Ag((X_{\infty,j})_{j=1}^\infty)]\forall g\in C_b(\mathbb{R}^\mathbb{N})\}:=E\\
        &\subseteq \mathcal{I}_\infty,
    \end{align*}
    where $C_b(X)$ denotes the continuous, bounded functions from $X$ to $\mathbb{R}$. $E$ is a Dynkin-system and $\bigcup_{t\geq 0}\mathcal{I}_t$ a $\pi$-system. By the $\pi$-$\lambda$ theorem, $\mathcal{I}_\infty\subseteq E\subseteq \mathcal{I}_{\infty}$ and thus $E=\mathcal{I}_\infty$.

    For each fixed $s$, by another characterization of stable convergence (see Proposition 5.33 of \cite{jacod_limit_2003}), it suffices to show $((X_{t,j})_{j=1}^\infty,Q)\overset{t\to\infty}{\Rightarrow}((X_{\infty,j})_{j=1}^\infty,Q)$ in distribution, for all r.v. $Q$ which are $\mathcal{I}_s$-measurable. In the proof above, take $T\geq s$ so that $Q$ is $\mathcal{I}_T$-measurable and insert $ir\cdot Q$ for any $r\in\mathbb{R}$ wherever we have a characteristic function.
\end{remark}

We now prove the main result.

\CLTProduct*
\begin{proof}
    We prove point (1), and add a remark regarding point (2). By Proposition \ref{prop:weak_topology_convergence}, it suffices to show a tightness condition for $(X_{t,j})_{j=1}^\infty$ in $\ell^p$. Fix $p= 2$. Firstly, it is clear that for each $t>0$, $X_t\in \ell^2$ almost surely. For $X_\infty$, we note that by using Corollary \ref{cor:sigmajk_large_jk}:
    \[\mathbb{E}[|X_{\infty,j}|^2]\lesssim j^{-1-\frac{\lambda_0}{\lambda_1}}.\]
    Therefore
    \[\mathbb{E}[\Vert X_\infty\Vert_2^2]=\sum_{j=1}^\infty \mathbb{E}[|X_{\infty,j}|^2]\leq  C\sum_{j=1}^\infty j^{-1-\lambda_0/(\lambda_1)}<\infty,\]
    so that $\Vert X_\infty\Vert_2<\infty$ almost surely. Define the following precompact set, in terms of a sequence $(r_k)_k$ and positive number $M$:
    \[K_{(r_k)_k,M}=\{x\in \ell^2:\Vert x\Vert_2^2\leq M,\forall k\in\mathbb{N},\sum_{j=r_k}^\infty |x_j|^2<1/k\},\]
    for ease of notation we write $K\doteq K_{(r_k)_k,M}$.
    Take $\varepsilon>0$, and we now show tightness. By performing a union bound:
    \[\mathbb{P}\left(X_t\not\in K\right)\leq \mathbb{P}(\Vert X_t\Vert_2^2\geq M)+\sum_{k=1}^\infty \mathbb{P}(\sum_{j=r_k}^\infty |X_{t,j}|^2\geq 1/k)\]
    The first term can be controlled by taking $M_\varepsilon=\varepsilon^{-1}$ with $\varepsilon\to 0^+$ after using Markov's inequality:
    \[\mathbb{P}(\Vert X_t\Vert_2^2\geq M_\varepsilon)\lesssim 1/M_{\varepsilon}=\varepsilon\]
    The upper bound is uniform in $t$ thanks to Remark \ref{rem:sigma_large_time_unif}. The second term, after an application of Markov's inequality, can be controlled by taking $r_k\sim \varepsilon^{-1}e^k$:
    \[\sum_{k=1}^\infty \mathbb{P}(\sum_{j=r_k}^\infty |X_{t,j}|^2\geq 1/k)\leq C''\sum_{k=1}^\infty k r_k^{-\lambda_0/\lambda_1}\leq C'''\varepsilon^{\lambda_0/\lambda_1}\sum_{k=1}^\infty ke^{-k\lambda_0/\lambda_1}\lesssim \varepsilon^{\lambda_0/\lambda_1},\]
    once again invoking Remark \ref{rem:sigma_large_time_unif}. With finite-dimensional projections being continuous from $(\ell^2,\Vert\cdot\Vert_2)\to \mathbb{R}^n$, using convergence in finite-dimensions, and from Prokhorov's theorem, we get convergence in $\ell^2$. To see convergence in $\ell^{p}$ for $2<p\leq \infty$, we can use continuous mapping with the (continuous) inclusion map:
    \[I:(\ell^2,\Vert\cdot\Vert_2)\to (\ell^p,\Vert\cdot\Vert_p).\]
\end{proof}
\begin{remark}
    By adapting Lemmas \ref{lem:sigma_large_time} and \ref{lem:control_fourth_moment} (replacing $j$ or $=j$ with $\geq j$), we can obtain the following factoids:
    \begin{align}
        &\text{CoV}[X_{t,i},X_{t,\geq 1}]\overset{t\to\infty}{\to} \sum_{k=1}^\infty \sigma_{i,k},\\
        &\mathbb{V}[X_{t,\geq 1}]\overset{t\to\infty}{\to} \sum_{k_1,k_2=1}^\infty \sigma_{k_1,k_2},\\
        &\sup_{t\geq 0}\mathbb{E}[|X_{t,\geq 1}|^4]<\infty.
    \end{align}
    Then, we can show \ref{eq:extended_clt} in finite dimensions by directly adapting Proposition \ref{prop:weak_topology_convergence}. Let $V_{t,1}=X_{t,\geq 1}$ and $V_{t,k}=X_{t,k-1}$ for $k\geq 2$, and so forth. Replacing $X$ with $V$ in the above proof of tightness to show Point 1 of Theorem \ref{thm:clt} shows tightness for $(V_{t,j})_{j=1}^\infty$, as desired.
    
\end{remark}

\subsubsection{Some easy corollaries and extensions}\label{sec:easy_cors}

We can now prove two simple corollaries

\CorrelationStructure*
\begin{proof}
    We can write for $k\neq j$:
    \begin{align}
        \frac{K(j,k)}{\sqrt{K(j,j)K(k,k)}}=\frac{\sigma_{j,k}}{\sqrt{\sigma_{j,j}\sigma_{k,k}}}=\frac{\frac{C}{jk}\left(\frac{1}{(j+k)^{\lambda_0/\lambda_1}}+\mathcal{O}_{j+k\to\infty}\left(\frac{1}{(j+k)^{\lambda_0/\lambda_1+1}}\right)\right)}{\sqrt{\left(C_2\frac{1}{j^{\lambda_0/\lambda_1+1}}+\mathcal{O}_{j\to\infty}(\frac{1}{j^{\lambda_0/\lambda_1+2}})\right)\left(C_2\frac{1}{k^{\lambda_0/\lambda_1+1}}+\mathcal{O}_{k\to\infty}(\frac{1}{k^{\lambda_0/\lambda_1+2}})\right)}}\label{eq:asymptotic_var}
    \end{align}
    Then, for sufficiently large $j,k\geq N$:
    \begin{align}
        \rho(j,k)\leq C\frac{1}{\sqrt{jk}}\frac{(jk)^{\lambda_0/(2\lambda_1)}}{(j+k)^{\lambda_0/\lambda_1}}\leq C\frac{1}{\sqrt{jk}},\label{eq:asymptotic_var_bound}
    \end{align}
    with the last step by the AM-GM inequality. Let $k(j)\neq j$ for $j$ large. Thus, using \ref{eq:asymptotic_var_bound}, $\rho(j,k(j))\to 0$ as $j\to\infty$.
\end{proof}

\CLTJt*
\begin{proof}
    Once we find that:
    \begin{align}
        \sigma^2:&=\lim_{t\to\infty}e^{-t\lambda_0}j(t)^{\frac{\lambda_1+\lambda_0}{\lambda_1}}\mathbb{V}\left[S_{j(t)}(t)-(av+\mu)Y\int_0^t P_{j(t)}(t-s)e^{\lambda_0 s}\,ds\right]\nonumber\\
        &=L,\label{eq:jt_sigma_L}
    \end{align}
    then the same proof as for Proposition \ref{prop:weak_topology_convergence} also proves the claim. It remains to show \ref{eq:jt_sigma_L}.
    
    The proof is identical to Lemma \ref{lem:sigma_large_time} in terms of how $\sigma^2$ is decomposed, but certain terms vanish as $P_{j(t)}(s)\leq \frac{1}{j(t)}$ for all $s$. In particular, using the law of total variance conditioning on $\mathcal{F}_\infty$:
    \begin{align}
    &e^{-\lambda_0t}j(t)^{\frac{\lambda_1+\lambda_0}{\lambda_1}}\mathbb{V}\left[S_{j(t)}(t)-(av+\mu)Y\int_0^t P_{j(t)}(t-s)e^{\lambda_0 s}\,ds\right]\label{eq:v_computation_jt}\\
    &=e^{-\lambda_0t}j(t)^{\frac{\lambda_1+\lambda_0}{\lambda_1}}\mathbb{E}\left[\sum_{T_i\leq t}P_{j(t)}(t-T_i)(1-P_{j(t)}(t-T_i))\right]\label{eq:v_inner}\\
    &+e^{-\lambda_0t}j(t)^{\frac{\lambda_1+\lambda_0}{\lambda_1}}\mathbb{V}\left[\sum_{T_i\leq t}P_{j(t)}(t-T_i)-(av+\mu)Y\int_0^t P_{j(t)}(t-s)e^{\lambda_0 s}\,ds\right]\label{eq:v_outer}
  \end{align}

  Note that \ref{eq:v_outer} tends to $0$ as $t\to\infty$ by using the same decomposition as in \ref{eq:term_total}. Namely letting $j=k=j(t)$, each of \ref{eq:term_Pj_R_Pk_R} to \ref{eq:term_Pk_R_Pj_ds} vanish by using the bound $P_{j(t)}\leq \frac{1}{j(t)}\to 0$ as $t\to\infty$. Therefore:
  \begin{align*}
    \sigma^2&=\lim_{t\to\infty}e^{-t\lambda_0}j(t)^{\frac{\lambda_1+\lambda_0}{\lambda_1}}\mathbb{E}\left[\sum_{T_i\leq t}P_{j(t)}(t-T_i)(1-P_{j(t)}(t-T_i))\right]\\
    &=\lim_{t\to\infty}e^{-t\lambda_0}j(t)^{\frac{\lambda_1+\lambda_0}{\lambda_1}}(av+\mu)\int_0^t P_{j(t)}(t-s)\left(1-P_{j(t)}(t-s)\right)e^{\lambda_0 s}\,ds=L.
  \end{align*}
\end{proof}

\subsection{Deviations in the limiting Gaussian process}\label{app:deviation_limiting}
The main result we prove here is Theorem \ref{thm:deviation_large_j}.
\DeviationLargeJ*
\begin{proof}
    Remark that $(X_{\infty,j})_{j=1}^\infty\overset{d}{=}(-X_{\infty,j})_{j=1}^\infty$. Thus, the result follows at once if we show both:
    \begin{align}
        \limsup_{j\to\infty}\frac{X_{\infty,j}}{\sqrt{2Y\sigma_{j,j}\log j}}\leq 1,\label{eq:show_upper_gauss}\\\limsup_{j\to\infty}\frac{X_{\infty,j}}{\sqrt{2Y\sigma_{j,j}\log j}}\geq 1.\label{eq:show_lower_gauss}
    \end{align}
\end{proof}

And, indeed we show both inequalities in two lemmas.

\begin{lemma}\label{lem:slepian}
    Let $(Q_j)_{j=1}^\infty$ be a standardized Gaussian process ($\mathbb{E}[Q_j]=0,\mathbb{E}[Q_j^2]=1$). Then:
    \begin{align}
        &\limsup_{j\to\infty}\frac{Q_j}{\sqrt{2\log j}}\leq 1,\label{eq:Q_upper_gaussian}
    \end{align}
    almost surely.
\end{lemma}
\begin{proof}
    Fix any $p>1$, then:
    \begin{align*}
        \mathbb{P}(Q_j>\sqrt{2p\log j})\leq Ke^{-\frac{1}{2}\sqrt{2p\log j}^2}=Kj^{-p},
    \end{align*}
    where $K>0$ is independent of $j$. As $j^{-p}$ is summable, we immediately obtain by Borel-Cantelli:
    \begin{align*}
        \limsup_{j\to\infty}\frac{Q_j}{\sqrt{2\log j}}\leq p.
    \end{align*}
    almost surely. Taking $p\to1^+$ along a discrete sequence concludes the result. 
\end{proof}

So, \ref{eq:show_upper_gauss} is shown using the standardization $(Y_j)_{j=1}^{\infty}:=(\frac{1}{\sqrt{Y\sigma_{j,j}}}X_{\infty,j})_{j=1}^\infty$. Showing \ref{eq:show_lower_gauss} is only a little more complicated.

\begin{lemma}
    \[\limsup_{j\to\infty}\frac{Y_j}{\sqrt{2\log j}}\geq 1\]
    almost surely.
\end{lemma}
\begin{proof}
    Fix $p$ within a small left neighborhood of $1$. Define the events $A_j=\{Y_j/\sqrt{2\log j}>p\}$. Define the sum $S_{j_1,j_2}=\sum_{k=j_1}^{j_2}\mathbf{1}_{A_k}$. We must show that $\{A_j\}_j$ occurs infinitely often. To do so, we claim that it suffices to prove that for every $j_1$ large:
    \begin{align}
        \lim_{j_2\to\infty}\mathbb{P}(S_{j_1,j_2}>0)= 1.\label{eq:Sj1j2}
    \end{align}
    Since $S_{j_1,j_2}$ is monotone in $j_2$, showing \ref{eq:Sj1j2} would prove that:
    \begin{align*}
        1=\mathbb{P}(\lim_{j_2\to\infty}S_{j_1,j_2}>0)=\mathbb{P}(\bigcup_{k=j_1}^\infty A_k),
    \end{align*}
    which would show $\{A_j\}_j$ occurs infinitely often.
    
    We first use the Paley-Zygmund inequality to get:
    \begin{align}
        \mathbb{P}(S_{j_1,j_2}>0)\geq\frac{(\mathbb{E}[S_{j_1,j_2}])^2}{\mathbb{E}[S_{j_1,j_2}^2]}.\label{eq:paley_zygmund}
    \end{align}
    We upper bound $\mathbb{E}[S_{j_1,j_2}^2]$ via:
    \begin{align}
        \mathbb{E}[S_{j_1,j_2}^2]&= \mathbb{E}[S_{j_1,j_2}]+2\sum_{j_1\leq i_1<i_2\leq j_2}\mathbb{P}(A_{i_1},A_{i_2})\nonumber\\
        &\leq \mathbb{E}[S_{j_1,j_2}]+\mathbb{E}[S_{j_1,j_2}]^2+2\sum_{j_1
        \leq i_1<i_2\leq j_2}|\mathbb{P}(A_{i_1},A_{i_2})-\mathbb{P}(A_{i_1})\mathbb{P}(A_{i_2})|\label{eq:upper_Sj1j2}.
    \end{align}
    The probability differences within the last sum in \ref{eq:upper_Sj1j2} are controlled via Eq. 4.2.2 of Section 4 of \cite{leadbetter_extremes_1983}. In particular:
    \begin{align}
        |\mathbb{P}(A_{i_1},A_{i_2})-\mathbb{P}(A_{i_1})\mathbb{P}(A_{i_2})|\leq \frac{1}{2\pi}|\text{CoV}[Y_{i_1},Y_{i_2}]|(1-|\text{CoV}[Y_{i_1},Y_{i_2}]|^2)^{-1/2}\exp\left(-\frac{p^2\log (i_1i_2)}{1+\text{CoV}[Y_{i_1},Y_{i_2}]}\right).\label{eq:summand_prob_bound}
    \end{align}
    Take $j_1$ so large such that $|\text{CoV}[Y_{i_1},Y_{i_2}]|=\rho(i_1,i_2)\leq \varepsilon<1$. We obtain therefore:
    \begin{align*}
        \ref{eq:summand_prob_bound}\leq C\frac{1}{\sqrt{i_1i_2}}(i_1i_2)^{-\frac{p^2}{1+\varepsilon}}.
    \end{align*}
    And so,
    \begin{align*}
        2\sum_{j_1
        \leq i_1<i_2\leq j_2}|\mathbb{P}(A_{i_1},A_{i_2})-\mathbb{P}(A_{i_1})\mathbb{P}(A_{i_2})|\leq C\left(\sum_{j_1\leq i<\infty}i^{-\frac{p^2}{1+\varepsilon}-\frac{1}{2}}\right)^2\leq K<\infty,
    \end{align*}
    as long as $p^2/(1+\varepsilon)+1/2>1$. In total, we then obtain:
    \begin{align*}
        \mathbb{E}[S_{j_1,j_2}^2]\leq \mathbb{E}[S_{j_1,j_2}]+\mathbb{E}[S_{j_1,j_2}]^2+K,
    \end{align*}
    with $K$ being independent of $j_1$. Continuing on with \ref{eq:paley_zygmund}:
    \begin{align*}
        \ref{eq:paley_zygmund}\geq \frac{1}{1+\mathbb{E}[S_{j_1,j_2}]^{-1}+K\mathbb{E}[S_{j_1,j_2}]^{-2}}.
    \end{align*}
    Thus, the conclusion follows so long as $\mathbb{E}[S_{j_1,j_2}]\to\infty$ as $j_2\to \infty$. But once again, using a Gaussian tail inequality:
    \begin{align*}
        \mathbb{P}(A_k)\geq K e^{-p^2\log k}/\sqrt{\log k}=K k^{-p^2}/\sqrt{\log k},
    \end{align*}
    which is not summable in $k$. So:
    \begin{align*}
        \mathbb{E}[S_{j_1,j_2}]=\sum_{k=j_1}^{j_2}\mathbb{P}(A_k)\overset{j_2\to\infty}{\to}\infty,
    \end{align*}
    proving the result, after taking $p\to 1^-$.
\end{proof}

\subsection{SFS central limits theorems at large detection sizes}\label{app:clt_detection_size}

Here we prove Theorem \ref{thm:clt_detection_size}. Throughout the proof of this theorem $\{M_n;n\geq 1\}$ is a positive sequence of numbers going to $0$ at a certain rate. Our strategy will be to approximate $\tau_n$ with an easier-to-work with stopping time $\tau_n'$. This will be the first time the early clones reach size $n$. Namely, define for Borel $B\subset\mathbb{R}$ the quantities $Z^{B}_1(t)=\sum_{T_i\in B}Z_1^{(i)}(t-T_i)$ and $W^B=\sum_{T_i\in B}e^{-\lambda_1 T_i}Y_1^{(i)}$. Then, $\tau_n':=\inf\{t\geq 0:Z_1^{[0,t_n]}(t)\geq n\}$. Here, $(t_n)_n$ is a sequence of numbers going to $\infty$. In preparation for this, we first note the following trivial convergence result for $Z_1^{[0,t_n]}$:

\begin{lemma}\label{lem:conv_Z_early}
    Let $(n_t)_{t\geq 0}\subseteq \mathbb{N}$ be a monotone non-decreasing net of numbers such that $n_t\to\infty$. Then:
    \begin{align*}
        e^{-\lambda_1 t}Z_1^{[0,n_t]}(t)\to W
    \end{align*}
    almost surely.
\end{lemma}
\begin{proof}
    Fix any $(r_m)_m$ be a sequence of numbers going to $\infty$. Notice that:
    \begin{align*}
        \sum_{i:T_i\leq r_m}e^{-\lambda_1 T_i}Y_1^{(i)}\leq\limsup_t e^{-\lambda_1 t}Z_1^{[0,n_t]}(t)\leq \limsup_s e^{-\lambda_1 t}Z_1(t)=W
    \end{align*}
    for all $m$ almost surely. Taking $m\to\infty$ concludes the result.
\end{proof}

Now, we obtain a crude large-deviation bound on $\tau_n$ and $\tau_n'$.

\begin{lemma}\label{lem:tau_n_in_range}
    For $A>0$ define $I_n^A:=[\frac{1}{\lambda_1}\ln n-A,\frac{1}{\lambda_1}\ln n+A]$, then:
    \begin{align*}
        &\lim_{A\to\infty}\mathbb{P}\left(\forall n,\,\tau_n\in I_n^A|\Omega_0^\infty\right)=1\\
        &\lim_{A\to\infty}\mathbb{P}\left(\forall n,\,\tau_n'\in I_n^A|\Omega_0^\infty\right)=1
    \end{align*}
\end{lemma}
\begin{proof}
    We first prove the result for $\tau_n$. It suffices to show that $\tau_n$ is outside the set $I_n^A$ with probability tending to $0$, and so -- defining $t_{n,+}^A=\frac{1}{\lambda_1}\ln n+A,t_{n,-}^A=\frac{1}{\lambda_1}\ln n-A$ -- by subadditivity it suffices to show:
    \begin{align}
        \mathbb{P}(\exists n,\,\tau_n\leq t_{n,-}^A|\Omega_0^\infty)\overset{A\to\infty}{\to}0\label{eq:tau_n_less_tnminus}\\
        \mathbb{P}(\exists n,\,\tau_n> t_{n,+}^A|\Omega_0^\infty)\overset{A\to\infty}{\to}0\label{eq:tau_n_more_tnplus}
    \end{align}

    Define $Z(t):=Z_0(t)+Z_1(t)$ as the total population. One can relate the time $\tau_n$ with a fixed time $s$ by:
    \begin{align}
        \{\tau_n\leq s\}=\{\sup_{t\in [0,s]}Z(t)\geq n\}.\label{eq:tau_n_compare}
    \end{align}

    Using \ref{eq:tau_n_compare}:
    \begin{align*}
        \mathbb{P}(\exists n,\,\tau_n\leq t_{n,-}^A|\Omega_0^\infty)=\mathbb{P}(\exists n,\,\sup_{t\in [0,t_{n,-}^A]}Z(t)\geq n|\Omega_0^\infty),
    \end{align*}
    then writing $ne^{-\lambda_1 t_{n,-}^A}=e^{\lambda_1 A}$:
    \begin{align*}
        &=\mathbb{P}(\exists n,\,e^{-\lambda_1 t_{n,-}^A}\sup_{t\in [0,t_{n,-}^A]}Z(t)\geq e^{\lambda_1 A}|\Omega_0^\infty)\\
        &\leq \mathbb{P}(\exists n,\,\sup_{t\in [0,t_{n,-}^A]}e^{-\lambda_1 t}Z(t)\geq e^{\lambda_1 A}|\Omega_0^\infty)\\
        &=\frac{1}{\mathbb{P}(\Omega_0^\infty)} \mathbb{P}(\sup_{t\in [0,\infty)}e^{-\lambda_1 t}Z(t)\geq e^{\lambda_1 A}).
    \end{align*}
    Recall that $e^{-\lambda_1 t}Z(t)$ converges almost surely to a random variable which is finite almost surely (e.g., see Lemma \ref{lem:W_form} or Theorem 3.2 in \cite{cheek2018}). Therefore $\sup_{t\in [0,\infty)}e^{-\lambda_1 t}Z(t)<\infty$ almost surely so that the above probability tends to $0$ as $A\to\infty$ by monotone convergence. Therefore, \ref{eq:tau_n_less_tnminus} is shown. Following the same series of equalities as above but replacing $t_{n,-}^A$ instead of $t_{n,+}^A$ and making appropriate adjustments:
    \begin{align*}
        \mathbb{P}(\exists n,\,\tau_n> t_{n,+}^A|\Omega_0^\infty)&=\mathbb{P}(\exists n,\,e^{-\lambda_1 t_{n,+}^A}\sup_{t\in [0,t_{n,+}^A]}Z(t)\leq e^{-\lambda_1 A}|\Omega_0^\infty)\\
        &\leq \mathbb{P}(\exists n,\,e^{-\lambda_1 t_{n,+}^A}Z(t_{n,+}^A)\leq e^{-\lambda_1 A}|\Omega_0^\infty).
    \end{align*}
    On $\Omega_0^\infty$, $e^{-\lambda_1 t}Z(t)$ converges almost surely to strictly positive random variable and $Z(t)>0$ for all $t$. By monotone convergence, the above tends to $0$ as $A\to\infty$, thereby showing \ref{eq:tau_n_more_tnplus}.

    For $\tau_n'$, using that $\tau_n\leq \tau_n'$:
    \begin{align*}
        \mathbb{P}(\exists n,\,\tau_n'\leq t_{n,-}^A|\Omega_0^\infty)\leq \mathbb{P}(\exists n,\,\tau_n\leq t_{n,-}^A|\Omega_0^\infty)\overset{A\to\infty}{\to}0.
    \end{align*}
    Now, in analog to Eq. \ref{eq:tau_n_compare}, we obtain:
    \begin{align*}
        \{\tau_n'\leq s\}=\{\sup_{t\in [0,s]}Z^{[0,t_n]}(t)\geq n\}.
    \end{align*}
    So for the lower bound:
    \begin{align*}
        \mathbb{P}(\exists n,\,\tau_n'> t_{n,-}^A|\Omega_0^\infty)\leq \mathbb{P}(\exists n,\,e^{-\lambda_1 t_{n,+}^A}Z^{[0,t_n]}(t_{n,+}^A)\leq e^{-\lambda_1 A}|\Omega_0^\infty)\overset{A\to\infty}{\to}0,
    \end{align*}
    by Lemma \ref{lem:conv_Z_early}.
\end{proof}

A crucial lemma in devising the approximation to $\tau_n$ will be to investigate how quickly the total population reaches its limit.
 Define the remainder terms $\zeta_t=e^{-\lambda_1 t}Z(t)-W$ and for $(t_n)_n\subseteq [0,\infty)$ an increasing sequence going to $\infty$ define  $\eta_t^n=e^{-\lambda_1 t}Z_1^{[0,t_n]}(t)-W^{[0,t_n]}$. 
\begin{lemma}\label{lem:remainder_control}
    Let $M_n\asymp n^{-\gamma}$ for $\gamma\in (0,\min\{1/2,(\lambda_1-\lambda_0)/\lambda_1\})$
    \begin{align*}
        &\mathbb{P}(|\zeta_{\tau_n}|>M_n|\Omega_0^\infty)\overset{n\to\infty}{\to} 0\\
        &\mathbb{P}(|\eta^n_{\tau_n'}|>M_n|\Omega_0^\infty)\overset{n\to\infty}{\to}0
    \end{align*}
\end{lemma}
\begin{proof}
    By Lemma \ref{lem:tau_n_in_range}, for fixed $\varepsilon>0$, there is $A$ large such that:
    \begin{align}
        &\mathbb{P}(|\zeta_{\tau_n}|>M_n|\Omega_0^\infty)\leq \frac{1}{\mathbb{P}(\Omega_0^\infty)}\mathbb{P}(\sup_{t\in I_n^A}|\zeta_t|>M_n)+\varepsilon,\label{eq:zeta_term_control}\\
        &\mathbb{P}(|\eta_{\tau_n'}^n|>M_n|\Omega_0^\infty)\leq \frac{1}{\mathbb{P}(\Omega_0^\infty)}\mathbb{P}(\sup_{t\in I_n^A}|\eta_t^n|>M_n)+\varepsilon,\label{eq:eta_term_control}
    \end{align}
    for all $n$ uniformly. Now, we may rewrite the contributions to $\zeta_t,\eta_t^n$ on a per-clone basis:
    \begin{align*}
        \zeta_t&=e^{-(\lambda_1-\lambda_0)t}e^{-\lambda_0 t}Z_0(t)+\sum_{T_i\leq t}e^{-\lambda_1 T_i}\left(e^{-\lambda_1(t-T_i)}Z_1^{(i)}(t-T_i)-Y_1^{(i)}\right)-\sum_{T_i>t}e^{-\lambda_1 T_i}Y_1^{(i)},\\
        \eta_t^n&=\sum_{T_i\leq \min\{t,t_n\}}e^{-\lambda_1 T_i}\left(e^{-\lambda_1 (t-T_i)}Z_1^{(i)}(t-T_i)-Y_1^{(i)}\right)-\sum_{t<T_i\leq t_n}e^{-\lambda_1 T_i}Y_1^{(i)}.
    \end{align*}

    By Markov inequality, to control the terms on the right-hand-side of \ref{eq:zeta_term_control} and \ref{eq:eta_term_control}:
    \begin{align}
        \mathbb{P}\left(\sup_{t\in I_n^A}|\zeta_t|>M_n\right)&\leq \frac{1}{M_n}e^{-(\lambda_1-\lambda_0)t_{n,-}^A}\mathbb{E}[\sup_{t\in I_n^A}e^{-\lambda_0 t}Z_0(t)]\label{eq:zeta_term_1}\\
        &+\frac{1}{M_n}\mathbb{E}\left[\sum_{T_i\leq t_{n,+}^A}e^{-\lambda_1 T_i}\mathbb{E}\left[\sup_{t\in [0,\infty)\cap (I_n^A-T_i)}\left|e^{-\lambda_1 t}Z_1^{(i)}(t)-Y_1^{(i)}\right||\mathcal{B}_\infty\right]\right]\label{eq:zeta_term_2}\\
        &+\frac{1}{M_n}\mathbb{E}\left[\sum_{T_i\geq t_{n,-}^A} e^{-\lambda_1 T_i}Y_1^{(i)}\right],\label{eq:zeta_term_3}\\
        \mathbb{P}\left(\sup_{t\in I_n^A}|\eta_t^n|>M_n\right)&\leq \frac{1}{M_n}   \mathbb{E}\left[\sum_{T_i\leq \min\{t_{n,+}^A,t_n\}}e^{-\lambda_1 T_i}\mathbb{E}\left[\sup_{t\in[0,\infty)\cap (I_n^A-T_i)}\left|e^{-\lambda_1 t}Z_1^{(i)}(t)-Y_1^{(i)}\right||\mathcal{B}_\infty\right]\right]\label{eq:eta_term_1}\\
        &+\frac{1}{M_n}\mathbb{E}\left[\sum_{t_n\geq T_i>t_{n,-}^A} e^{-\lambda_1 T_i}Y_1^{(i)}\right].\label{eq:eta_term_2}
    \end{align}

    It remains to control each term. We can use Doob's maximal inequality for the right-hand-side of \ref{eq:zeta_term_1}:
    \begin{align*}
        e^{-(\lambda_1-\lambda_0)t_{n,-}^A}\mathbb{E}[\sup_{t\in I_n^A}e^{-\lambda_0 t}Z_0(t)]\leq e^{-(\lambda_1-\lambda_0)t_{n,-}^A}\sqrt{4\mathbb{E}\left[\left(e^{-\lambda_0 t_{n,+}^A}Z_0(t_{n,+}^A)\right)^2\right]}=\mathcal{O}(n^{-\frac{\lambda_1-\lambda_0}{\lambda_1}}).
    \end{align*}

    We can use a simple application of Campbell's theorem for the right-hand-side of \ref{eq:zeta_term_3}:
    \begin{align*}
        \mathbb{E}\left[\sum_{T_i\geq t_{n,-}^A} e^{-\lambda_1 T_i}Y_1^{(i)}\right]=\mathbb{E}\left[\sum_{T_i\geq t_{n,-}^A} e^{-\lambda_1 T_i}\mathbb{E}[Y_1^{(i)}|\mathcal{B}_\infty]\right]&=(av+\mu)\int_{t_{n,-}^A}^\infty e^{-\lambda_1 s}e^{\lambda_0 s}\,ds\\
        &=\mathcal{O}\left(n^{-\frac{\lambda_1-\lambda_0}{\lambda_1}}\right).
    \end{align*}

    We can trivially bound term \ref{eq:eta_term_1} by term \ref{eq:zeta_term_2} and term \ref{eq:eta_term_2} by \ref{eq:zeta_term_3}. It then remains to control \ref{eq:zeta_term_2}. We first control the inner  $\mathbb{E}[\cdot|\mathcal{B}_\infty]$ term by using triangle inequality:
    \begin{align}
        &\mathbb{E}\left[\sup_{t\in[0,\infty)\cap (I_n^A-T_i)}\left|e^{-\lambda_1 t}Z_1^{(i)}(t)-Y_1^{(i)}\right||\mathcal{B}_\infty\right]\\
        &\leq \mathbb{E}\left[\sup_{t\in [0,\infty)\cap (I_n^A-T_i)}|e^{-\lambda_1 t}Z_1^{(i)}(t)-e^{-\lambda_1 (t_{n,-}^A-T_i)^+}Z_1^{(i)}((t_{n,-}^A-T_i)^+)||\mathcal{B}_\infty\right]\label{eq:term_clone_1}\\
    &+\mathbb{E}\left[\left|e^{-\lambda_1 (t_{n,-}^A-T_i)^+}Z_1^{(i)}((t_{n,-}^A-T_i)^+)-Y_1^{(i)}\right||\mathcal{B}_\infty\right].\label{eq:term_clone_2}
    \end{align}

    Term \ref{eq:term_clone_1} is controlled by Doob's submartingale maximal inequality, and term \ref{eq:term_clone_2} is controlled by using the $L^2$ rate of convergence of a scaled one-type branching process to its limit (see Lemma 7 of [main])
    \begin{align*}
        &\leq 2\sqrt{\mathbb{E}\left[|e^{-\lambda_1 (t_{n,+}^A-T_i)}Z_1^{(i)}(t_{n,+}^A-T_i)-e^{-\lambda_1 (t_{n,-}^A-T_i)^+}Z_1^{(i)}((t_{n,-}^A-T_i)^+)|^2|\mathcal{B}_\infty\right]}\\
        &+Ce^{-\frac{\lambda_1}{2} (t_{n,-}^A-T_i)^+}\\
        &\leq 4\sqrt{\mathbb{E}\left[|e^{-\lambda_1 (t_{n,+}^A-T_i)}Z_1^{(i)}(t_{n,+}^A-T_i)-Y_1^{(i)}|^2|\mathcal{B}_\infty\right]+\mathbb{E}\left[|Y_1^{(i)}-e^{-\lambda_1 (t_{n,-}^A-T_i)^+}Z_1^{(i)}((t_{n,-}^A-T_i)^+)|^2|\mathcal{B}_\infty\right]}\\
        &+Ce^{-\frac{\lambda_1}{2} (t_{n,-}^A-T_i)^+}\\
        &\leq C e^{-\frac{\lambda_1}{2}(t_{n,-}^A-T_i)^+}\leq Ce^{-\frac{\lambda_1}{2}(t_{n,+}^A-T_i)^+}.
    \end{align*}
    The last inequality is from noticing $t_{n,+}^A=t_{n,-}^A+2A$ and modifying the constant $C$ by a factor of $e^{\lambda_1 A}$. Let $\gamma'=(\lambda_1-\lambda_0)/\lambda_1$ if $\lambda_1/2<\lambda_0$ and $\gamma'=1/2$ otherwise. Then, using this bound with Campbell's Theorem:
    \begin{align*}
        &\mathbb{E}\left[\sum_{T_i\leq t_{n,+}^A}e^{-\lambda_1 T_i}\mathbb{E}\left[\sup_{t\in [0,\infty)\cap (I_n^A-T_i)}\left|e^{-\lambda_1 t}Z_1^{(i)}(t)-Y_1^{(i)}\right||\mathcal{B}_\infty\right]\right]\\
        &\leq C\mathbb{E}\left[\sum_{T_i\leq t_{n,+}^A}e^{-\lambda_1 T_i}e^{-\frac{\lambda_1}{2}(t_{n,+}^A-T_i)}\right]\\
        &=Ce^{-\frac{\lambda_1}{2}t_{n,+}^A}\int_0^{t_{n,+}^A}e^{-\frac{\lambda_1}{2}s}e^{\lambda_0 s}\,ds=\mathcal{O}\left(n^{-\gamma'}\right)
    \end{align*}
    
\end{proof}

Now, we can construct the approximation to $\tau_n$ by another stopping time.

\begin{lemma}\label{lem:stopping_time_approximation}
  For $\beta>1$ and $\gamma$ as in Lemma \ref{lem:remainder_control}, let $M_n\asymp n^{-\gamma}$ and set $t_n=\frac{\beta}{\lambda_1-\lambda_0}\ln(1/M_n)$. Then:
    \begin{align*}
        \mathbb{P}(|\tau_n-\tau_n'|>M_n|\Omega_0^\infty)\overset{n\to\infty}{\to}0.
    \end{align*}
\end{lemma}
\begin{proof}
    Define the intermediate random times $\tau_n^{(1)}=\frac{1}{\lambda_1}\ln(n/W)$ and $\tau_n^{(2)}=\frac{1}{\lambda_1}\ln(n/W^{[0,t_n]})$. Then by a union bound:
    \begin{align}
        \mathbb{P}(|\tau_n-\tau_n'|>M_n|\Omega_0^\infty)&\leq \mathbb{P}(|\tau_n-\tau_n^{(1)}|>M_n/3|\Omega_0^\infty)\label{eq:tau_n_tau_n_1}\\
        &+\mathbb{P}(|\tau_n^{(1)}-\tau_n^{(2)}|>M_n/3|\Omega_0^\infty)\label{eq:tau_n_1_tau_n_2}\\
        &+\mathbb{P}(|\tau_n^{(2)}-\tau_n'|>M_n/3|\Omega_0^\infty).\label{eq:tau_n_2_tau_n_prime}
    \end{align}
Using the definitions of $\zeta_t,\eta_t^n$, we can write:
    \begin{align*}
        |\tau_n-\tau_n^{(1)}|=\frac{1}{\lambda_1}\left|\ln\left(1+\frac{\zeta_{\tau_n}}{W}\right)\right|,\,|\tau_n'-\tau_n^{(2)}|=\frac{1}{\lambda_1}\left|\ln\left(1+\frac{\eta_{\tau_n'}^n}{W^{[0,t_n]}}\right)\right|.
    \end{align*}
    Since $W$ and $W^{[0,t_n]}$ are both in the denominator in the previous display, we need to control the probability that they are close to zero. In particular, on $\Omega_0^\infty$, $W>0$ almost surely so that for any fixed $\varepsilon>0$ there is a $\delta>0$ with $\mathbb{P}(W\geq \delta|\Omega_0^\infty)\geq 1-\varepsilon$. Notice that $W^{[0,t_n]}$ is monotone in $n$ with $\lim_{n\to\infty}W^{[0,t_n]}=W>0$ on $\Omega_0^\infty$. Therefore, there is some $n_0$ such that $n\geq n_0$ implies $\mathbb{P}(W^{[0,t_n]}>0)\geq 1-\varepsilon/2$. Then, there is a $\delta_2>0$ such that for all $n\geq n_0$, $\mathbb{P}(W^{[0,t_n]}\geq \delta_2)\geq 1-\varepsilon$. We can control the right-hand-sides of terms \ref{eq:tau_n_tau_n_1} and \ref{eq:tau_n_2_tau_n_prime}.
    
    To control term \ref{eq:tau_n_tau_n_1} we use the bound $|\ln(1+x)|\leq c|x|$ for $x$ in a neighborhood of $0$. In particular,
    \begin{align*}
        \mathbb{P}(|\tau_n-\tau_n^{(1)}|>M_n/3|\Omega_0^\infty)&\leq \mathbb{P}\left(\frac{1}{\lambda_1}\left|\ln\left(1+\frac{\zeta_{\tau_n}}{W}\right)\right|>M_n/3,|\zeta_{\tau_n}|\leq M_n,W\geq \delta|\Omega_0^\infty\right)+\varepsilon\\
        &\leq \mathbb{P}\left(\left|\zeta_{\tau_n}\right|>\delta\lambda_1 M_n/(3c),|\zeta_{\tau_n}|\leq M_n,W\geq \delta|\Omega_0^\infty\right)+\varepsilon\overset{n\to\infty}{\to}\varepsilon.
    \end{align*}
We can use a similar argument to control term \ref{eq:tau_n_2_tau_n_prime}
    \begin{align*}
        \limsup_{n\to\infty}\mathbb{P}(|\tau_n'-\tau_n^{(2)}|>M_n/3|\Omega_0^\infty)\leq \varepsilon.
    \end{align*}
It thus remains to control \ref{eq:tau_n_1_tau_n_2}. We have:
    \begin{align*}
        |\tau_n^{(1)}-\tau_n^{(2)}|\leq -\frac{1}{\lambda_1}\ln (W^{[0,t_n]}/W)=-\frac{1}{\lambda_1}\ln\left(1+\frac{W^{[0,t_n]}-W}{W}\right).
    \end{align*}
    Thus, similarly to controlling the previous terms:
    \begin{align*}
        \mathbb{P}(|\tau_n^{(1)}-\tau_n^{(2)}|>M_n/3|\Omega_0^\infty)\leq \mathbb{P}(W-W^{[0,t_n]}>\delta \lambda_1 M_n/(3c)|\Omega_0^\infty)+\varepsilon.
    \end{align*}
    and using Markov's inequality followed by Campbell's theorem:
    \begin{align*}
        &\leq \frac{3c}{\delta \lambda_1 M_n}\mathbb{E}\left[\sum_{T_i>t_n}e^{-\lambda_1 T_i}Y_1^{(i)}\right]+\varepsilon\\
        &\leq \frac{C}{M_n}e^{-(\lambda_1-\lambda_0)t_n}+\varepsilon\overset{n\to\infty}{\to}\varepsilon.
    \end{align*}
\end{proof}

We then show that to get the limit of $X_{\tau_n,j}$ in distribution, it suffices to examine the limit of $X_{\tau_n',j}$. To do this, we need the following lemma that $X_{\tau_n,j}-X_{\tau_n',j}$ converge to $0$ in probability conditional on $\Omega_0^\infty$.
\begin{lemma}\label{lem:tau_n_tau_n_prime_X}
    Let $t_n,M_n$ be as in Lemma \ref{lem:stopping_time_approximation}. Then, for all $\varepsilon>0$:
    \begin{align*}
        \lim_{n\to\infty}\mathbb{P}(|X_{\tau_n,j}-X_{\tau_n',j}|>\varepsilon,\Omega_0^\infty)= 0
    \end{align*}
\end{lemma}
\begin{proof}
    We begin by decomposing $X_{t,j}$ as follows:
    \begin{align}
        &X_{t,j}=X_{t,j}^{(1)}+X_{t,j}^{(2)}+X_{t,j}^{(3)},\label{eq:X_t_decomposition}\\
        &X_{t,j}^{(1)}=e^{-\lambda_0t/2}(S_j(t)-\sum_{T_i\leq t}P_j(t-T_i)),\nonumber\\
        &X_{t,j}^{(2)}=e^{-\lambda_0 t/2}\int_0^t P_j(t-s)\,dR_s,\nonumber\\
        &X_{t,j}^{(3)}=e^{-\lambda_0 t/2}(av+\mu)\int_0^t P_j(t-s)(Z_0(s)-e^{\lambda_0 s}Y)\,ds.\nonumber
    \end{align}
    It suffices to show that for any fixed $\varepsilon>0$, $\mathbb{P}(|X_{\tau_n,j}^{(i)}-X_{\tau_n',j}^{(i)}|>\varepsilon,\Omega_0^\infty)\to 0$.

    We begin with the $i=1$ term. We further subdivide the term $X_{t,j}^{(1)}$ into a martingale and integral component. To do this, consider $Z_1^{(i)}$ as a one-type continuous-time linear birth-death process, and let $\mathcal{A}$ denote its infinitesimal generator. Define the function $f_j(x)=\mathbf{1}_{x=j}$. Then:
    \begin{align*}
        (\mathcal{A}f_j)(x)&=(a+b)xf_j(x+1)+dxf_j(x-1)-(a+b+d)xf_j(x)\\
        &=(a+b)(j-1)\mathbf{1}_{x=j-1}+d(j+1)\mathbf{1}_{x=j+1}-(a+b+d)j\mathbf{1}_{x=j}
    \end{align*}

    By the Dynkin martingale formula (see Proposition 1.7 in Chapter 4, \cite{ethier_kurtz_ch4}), we observe that
    \begin{align}
        m_j^{(i)}(t):&=\mathbf{1}_{Z_1^{(i)}(t)=j}-\mathbf{1}_{1=j}-\int_0^t\mathcal{A}f_j(Z_1^{(i)}(s))\,ds\label{eq:m_ji_martingale}\\
        &=\mathbf{1}_{Z_1^{(i)}(t)=j}-\mathbf{1}_{1=j}-\int_0^t(a+b)(j-1)\mathbf{1}_{Z_1^{(i)}(s)=j-1}+d(j+1)\mathbf{1}_{Z_1^{(i)}(s)=j+1}-(a+b+d)j\mathbf{1}_{Z_1^{(i)}(s)=j}\,ds\nonumber
    \end{align}
    is a 0-mean martingale with respect to the filtration generated by $Z_1^{(i)}$. Therefore, for all $t\geq 0$, we have the following relation:
    \[0=\mathbb{E}[m_j^{(i)}(t)]=P_j(t)-\mathbf{1}_{1=j}-\int_0^t \mathbb{E}[\mathcal{A}f_j(Z_1^{(i)}(s))]\,ds.\]
    Combining this with \ref{eq:m_ji_martingale} yields the equality
    \begin{align*}
        \mathbf{1}_{Z_1^{(i)}(t)=j}-P_j(t)=m_j^{(i)}(t)+\int_0^t \mathcal{A}f(Z_1^{(i)}(s))-\mathbb{E}[\mathcal{A}f_j(Z_1^{(i)}(s))]\,ds.
    \end{align*}
    Substituting $t$ with $t-T_i$ and summing over $T_i\leq t$ gives the decomposition we will examine:
    \begin{align}
        &e^{\frac{\lambda_0}{2}t}X_{t,j}^{(1)}=\sum_{T_i\leq t}(\mathbf{1}_{Z_1^{(i)}(t-T_i)=j}-P_j(t-T_i))=\Gamma_{1,j}(t)+\Gamma_{2,j}(t)\label{eq:X1_decompose_gamma},\\
        &\Gamma_{1,j}(t)=\sum_{T_i\leq t}m_j^{(i)}(t-T_i),\\
        &\Gamma_{2,j}(t)=\sum_{T_i\leq t}\int_0^{t-T_i} \mathcal{A}f(Z_1^{(i)}(s))-\mathbb{E}[\mathcal{A}f(Z_1^{(i)}(s))]\,ds.
    \end{align}
    Define $N_i(t)=m_j^{(i)}\left((t-T_i)^+\right)$, and note that because $T_i$ is a stopping time and the clone evolution after $T_i$ is independent of the past the shifted process $N_i$ is still an $\mathcal{I}_t$-martingale. Furthermore we can write
    \[
    \Gamma_{1,j}(t)=\sum_{i=1}^\infty N_i(t),
    \]
    which has finite $L^1$-norm for each $t\geq 0$. Thus, since $\Gamma_{1,j}$ is a sum of martingales, it to is a $\mathcal{I}_t$-martingale.
    
    Finally, we may compare $X_{\tau_n,j}^{(1)}$ and $X_{\tau_n',j}^{(1)}$. By Lemma \ref{lem:stopping_time_approximation}, fix any $\delta>0$ and large $n$ such that $\mathbb{P}(|\tau_n'-\tau_n|\leq M_n,\Omega_0^\infty)\leq \delta$ so that:
    \begin{align*}
        \mathbb{P}(|X_{\tau_n,j}^{(1)}-X_{\tau_n',j}^{(1)}|>\varepsilon,\Omega_0^\infty)\leq \mathbb{P}(\sup_{s\in [0,M_n]}|X_{\tau_n+s,j}^{(1)}-X_{\tau_n,j}^{(1)}|>\varepsilon,|\tau_n'-\tau_n|\leq M_n,\Omega_0^\infty)+\delta,
    \end{align*}
    where we ignore the region $[-M_n,0)$ as $\tau_n\leq \tau_n'$ almost surely. Note that:
    \begin{align*}
        X_{t,j}^{(1)}-X_{s,j}^{(1)}&=\left(e^{-\frac{\lambda_0}{2}t}-e^{-\frac{\lambda_0}{2}s}\right)\left(S_j(t)-\sum_{T_i\leq t}P_j(t-T_i)\right)\\
        &+e^{-\frac{\lambda_0 }{2}s}\left(S_j(t)-\sum_{T_i\leq t}P_j(t-T_i)-\left(S_j(s)-\sum_{T_i\leq s}P_j(s-T_i)\right)\right).
    \end{align*}
    Then, using that $|1-e^{-s}|\leq 2s$ for $s$ in a neighborhood of $0$:
    \begin{align}
        &\mathbb{P}(\sup_{s\in [0,M_n]}|X_{\tau_n+s,j}^{(1)}-X_{\tau_n,j}^{(1)}|>\varepsilon,\Omega_0^\infty)\nonumber\\
        &\leq \mathbb{P}(e^{-\frac{\lambda_0}{2} \tau_n}M_n\sup_{s\in [0,M_n]}\left|S_j(\tau_n+s)-\sum_{T_i\leq \tau_n+s}P_j(\tau_n+s-T_i)\right|>\varepsilon/4,\Omega_0^\infty)\label{eq:1_control_part_1}\\
        &+\mathbb{P}(e^{-\frac{\lambda_0 }{2}\tau_n}\sup_{s\in [0,M_n]}\left|S_j(\tau_n+s)-\sum_{T_i\leq \tau_n+s}P_j(\tau_n+s-T_i)-\left(S_j(\tau_n)-\sum_{T_i\leq \tau_n}P_j(\tau_n-T_i)\right)\right|>\varepsilon/2,\Omega_0^\infty).\label{eq:1_control_part_2}
    \end{align}
    We further decompose the term in the suprema for both \ref{eq:1_control_part_1} and \ref{eq:1_control_part_2} by $\Gamma_{1,j},\Gamma_{2,j}$. Observe that by further decomposing based on the events $\tau_n\in I_n^A$ and $\tau_n\notin I_n^A$:
    \begin{align*}
        \ref{eq:1_control_part_1}-\delta&\leq \mathbb{P}(e^{-\frac{\lambda_0}{2} t_{n,-}^A}M_n\sup_{s\in I_n^A}\left|S_j(s)-\sum_{T_i\leq s}P_j(s-T_i)\right|>\varepsilon/4)\\
        &\leq\mathbb{P}(e^{-\frac{\lambda_0}{2} t_{n,-}^A}M_n\sup_{s\in I_n^A}\left|\Gamma_{1,j}(s)\right|>\varepsilon/8)\\
        &+\mathbb{P}(e^{-\frac{\lambda_0}{2} t_{n,-}^A}M_n\sup_{s\in I_n^A}\left|\Gamma_{2,j}(s)\right|>\varepsilon/8).
    \end{align*}
    Since $\Gamma_{1,j}$ is a martingale, we apply Doob's maximal inequality to get:
    \begin{align*}
        &\mathbb{P}(e^{-\frac{\lambda_0}{2} t_{n,-}^A}M_n\sup_{s\in I_n^A}\left|\Gamma_{1,j}(s)\right|>\varepsilon/8)\\
        &\leq \frac{8}{\varepsilon}e^{-\frac{\lambda_0}{2}t_{n,-}^A}M_n\sqrt{\mathbb{E}[|\Gamma_{1,j}(t_{n,+}^A)|^2]}.
    \end{align*}
    The quadratic variation of $\Gamma_{1,j}(t)$ is dominated by $\sum_{T_i\leq t}E_j^{(i)}$. $E_j^{(i)}$ counts the total number of entrances and exits of clone $i$ from size $j$ over all time. Remark that $E_j^{(i)}$ is dominated by a geometric random variable so that it has finite moment (say smaller than $C$). Then:
    \[\mathbb{E}[|\Gamma_{1,j}(t_{n,+}^A)|^2]=\mathbb{E}[[\Gamma_{1,j}]_{t_{n,+}^A}]\leq \mathbb{E}[\sum_{T_i\leq t_{n,+}^A}C]\leq C_2 n^{\frac{\lambda_0}{\lambda_1}}.\]
    Thus:
    \[\mathbb{P}(e^{-\frac{\lambda_0}{2} t_{n,-}^A}M_n\sup_{s\in I_n^A}\left|\Gamma_{1,j}(s)\right|>\varepsilon/8)\overset{n\to\infty}{\to} 0\]
    To control the supremum term involving $\Gamma_{2,j}$, we can rewrite $\Gamma_{2,j}$ by performing a substitution $s=u-T_i$ to observe:
    \begin{align*}
        \Gamma_{2,j}(t)&=\int_{0}^{t} \sum_{T_i\leq t}\mathbf{1}_{u\geq T_i}\left(\mathcal{A}f(Z_1^{(i)}(u-T_i))-\mathbb{E}[\mathcal{A}f(Z_1^{(i)}(s))]|_{s=u-T_i}\right)\,du\\
        &=\int_{0}^{t} \sum_{T_i\leq u}\left(\mathcal{A}f(Z_1^{(i)}(u-T_i))-\mathbb{E}[\mathcal{A}f(Z_1^{(i)}(s))]|_{s=u-T_i}\right)\,du\\
        &=\int_{0}^{t} (a+b)(j-1)\sum_{T_i\leq u}\left(\mathbf{1}_{Z_1^{(i)}(u-T_i)=j-1}-P_{j-1}(u-T_i)\right)\\
        &+d(j+1)\sum_{T_i\leq u}\left(\mathbf{1}_{Z_1^{(i)}(u-T_i)=j+1}-P_{j+1}(u-T_i)\right)\\
        &-(a+b+d)j\sum_{T_i\leq u}\left(\mathbf{1}_{Z_1^{(i)}(u-T_i)=j}-P_{j}(u-T_i)\right)\,du
    \end{align*}

    Now:
    \begin{align}
        &\mathbb{P}(e^{-\frac{\lambda_0}{2} t_{n,-}^A}M_n\sup_{s\in I_n^A}\left|\Gamma_{2,j}(s)\right|>\varepsilon/8)\nonumber\\
        &\leq\mathbb{P}\biggl(e^{-\frac{\lambda_0}{2} t_{n,-}^A}M_n\int_{0}^{t_{n,+}^A} (a+b)(j-1)\left|\sum_{T_i\leq u}\mathbf{1}_{Z_1^{(i)}(u-T_i)=j-1}-P_{j-1}(u-T_i)\right|\nonumber\\
        &+d(j+1)\left|\sum_{T_i\leq u}\mathbf{1}_{Z_1^{(i)}(u-T_i)=j+1}-P_{j+1}(u-T_i)\right|\nonumber\\
        &-(a+b+d)j\left|\sum_{T_i\leq u}\mathbf{1}_{Z_1^{(i)}(u-T_i)=j}-P_{j}(u-T_i)\right|\,du>\varepsilon/8\biggr)\label{eq:gamma_2_part_continue}.
    \end{align}
    Conditional on $\mathcal{B}_\infty$, $\mathbf{1}_{Z_1^{(i_1)}(u-T_{i_1})=j}$ and $\mathbf{1}_{Z_1^{(i_2)}(u-T_{i_2})=j}$ are uncorrelated. Thus:
    \begin{align*}
        \mathbb{E}\left[\left|\sum_{T_i\leq u}\mathbf{1}_{Z_1^{(i)}(u-T_i)=j}-P_{j}(u-T_i)\right|^2\right]=\mathbb{E}\left[\sum_{T_i\leq u}P_j(u-T_i)(1-P_j(u-T_i))\right]\leq Ce^{\lambda_0 u}.
    \end{align*}
    Then, by Markov's inequality:
    \begin{align*}
        \ref{eq:gamma_2_part_continue}\leq \frac{8}{\varepsilon}e^{-\frac{\lambda_0}{2}t_{n,-}^A}M_n\int_0^{t_{n,+}^A}C_2 e^{\frac{\lambda_0}{2}u}\,du=\mathcal{O}(M_n)\overset{n\to\infty}{\to}0.
    \end{align*}
    Finally, we conclude that the right-hand-side of \ref{eq:1_control_part_1} tends to $0$. 
    
    For term \ref{eq:1_control_part_2}, we first apply Lemma \ref{lem:tau_n_in_range}, the triangle inequality and a union bound to obtain:
    \begin{align}
        \ref{eq:1_control_part_2}-\delta&\leq \mathbb{P}(n^{-\frac{\lambda_0}{2\lambda_1}}\sup_{s\in [0,M_n]}|\Gamma_{1,j}(\tau_n+s)-\Gamma_{1,j}(\tau_n)|>\varepsilon/(4C),\tau_n\leq t_{n,+}^A,\Omega_0^\infty)\label{eq:1_control_2_1}\\
        &+\mathbb{P}(n^{-\frac{\lambda_0}{2\lambda_1}}\sup_{s\in [0,M_n]}|\Gamma_{2,j}(\tau_n+s)-\Gamma_{2,j}(\tau_n)|>\varepsilon/(4C),\tau_n\leq t_{n,+}^A,\Omega_0^\infty).\label{eq:1_control_2_2}
    \end{align}
    Note that on $\tau_n\leq t_{n,+}^A$, $\tau_n\wedge (t_{n,+}^A)=:\tau_n''=\tau_n$ almost surely. $\tau_n''$ is a bounded $\mathcal{I}_t$-stopping time. Then, by optional stopping, $\Gamma_{1,j}(\tau_n''+\cdot)-\Gamma_{1,j}(\tau_n'')$ is a $\mathcal{I}_{\tau_n''+\cdot}$-martingale. Hence we may use Doob's maximal inequality to control \ref{eq:1_control_2_1}:
    \begin{align}
        \ref{eq:1_control_2_1}&\leq \mathbb{P}(n^{-\frac{\lambda_0}{2\lambda_1}}\sup_{s\in [0,M_n]}|\Gamma_{1,j}(\tau_n''+s)-\Gamma_{1,j}(\tau_n'')|>\varepsilon/(4C))\nonumber\\
        &\leq \frac{4C}{\varepsilon}n^{-\frac{\lambda_0}{2\lambda_1}}\sqrt{\mathbb{E}\left[\left|\Gamma_{1,j}(\tau_n''+M_n)-\Gamma_{1,j}(\tau_n'')\right|^2\right]}.\label{eq:upper_1_control_2_1}
    \end{align}
    Let $E_j^{(i)}(u)$ denote the number of entrances and exists clone the $i$th clone takes into and out of size $j$ $u$ time units after inception. Therefore $E_j^{(i)}=E_j^{(i)}(\infty)$. The quadratic variation of $\Gamma_{1,j}(\tau_n''+M_n)-\Gamma_{1,j}(\tau_n'')$ is then dominated by:
    \begin{align*}
        [\Gamma_{1,j}(\tau_n''+\cdot)-\Gamma_{1,j}(\tau_n'')]_{M_n}\leq \sum_{T_i\leq \tau_n''}(E_j^{(i)}(\tau_n''+M_n-T_i)-E_j^{(i)}(\tau_n''-T_i))+\sum_{T_i\in [\tau_n'',\tau_n''+M_n]}E_j^{(i)}.
    \end{align*}
    We control the first term by the rarity of having type-1 entrances and exits occur during an $\mathcal{O}(M_n)$-window of time, and we control the second term by the rarity of having mutation events occur during an $\mathcal{O}(M_n)$-window of time. For the first term, we note that by examining the mean of $E_j^{(i)}(s)$ we obtain the uniform bound:
    \begin{align*}
        \sup_{k\geq 1}\mathbb{E}[E_j^{(1)}(s)|Z_1^{(1)}(0)=k]=\sup_{k\geq 1}\int_0^s d(j+1)P_{j+1}(s')+(a+b)(j-1)P_{j-1}(s')\,ds'\leq C_js
    \end{align*}
    
    Then, by using the Strong Markov property, we can control the first term:

    \begin{align*}
         \mathbb{E}\left[\sum_{T_i\leq \tau_n''}(E_j^{(i)}(\tau_n''+M_n-T_i)-E_j^{(i)}(\tau_n''-T_i))|\mathcal{I}_{\tau_n''}\right]&\leq C_j\sum_{T_i\leq \tau_n''}M_n\\
         &=CM_nM(\tau_n'').
    \end{align*}
    Taking expectation:
    \begin{align}
        \mathbb{E}\left[\sum_{T_i\leq \tau_n''}(E_j^{(i)}(\tau_n''+M_n-T_i)-E_j^{(i)}(\tau_n''-T_i))\right]\leq CM_n\mathbb{E}[M(t_{n,+}^A)]=CM_ne^{\lambda_0 t_{n,+}^A}.\label{eq:expectation_upper_1_control_2_0}
    \end{align}

    For the second term, once again using the strong Markov property:
    \begin{align}
        \mathbb{E}\left[\sum_{T_i\in [\tau_n'',\tau_n''+M_n]}E_j^{(i)}\right]&=\mathbb{E}\left[\mathbb{E}_{Z_0(\tau_n'')}\left[\sum_{T_i\leq M_n}E_j^{(i)}\right]\right]\nonumber\\
        &=(av+\mu)\mathbb{E}[E_j^{(1)}]\int_0^{M_n}\mathbb{E}\left[Z_0(\tau_n'')\right]e^{\lambda_0 s}\,ds\nonumber\\
        &\leq (av+\mu)\mathbb{E}[E_j^{(1)}]\int_0^{M_n}\mathbb{E}\left[\sup_{t\in [0,\infty)}(Z_0(t)e^{-\lambda_0 t})\right]e^{\lambda_0 (s+t_{n,+}^A)}\,ds\nonumber\\
        &\leq C M_n e^{\lambda_0 t_{n,+}^A}.\label{eq:expectation_upper_1_control_2_1}
    \end{align}
    Plugging \ref{eq:expectation_upper_1_control_2_1} and \ref{eq:expectation_upper_1_control_2_0} into \ref{eq:upper_1_control_2_1}, \ref{eq:1_control_2_1} tends to $0$. Regarding term \ref{eq:1_control_2_2}, we employ Markov's inequality with Cauchy-Schwarz to see:
    \begin{align*}
        \ref{eq:1_control_2_2}&\leq \mathbb{P}\left(n^{-\frac{\lambda_0}{2\lambda_1}}\int_{\tau_n''}^{\tau_n''+M_n}\left|\sum_{T_i\leq u}\left(\mathcal{A}f(Z_1^{(i)}(u-T_i))-\mathbb{E}[\mathcal{A}f(Z_1^{(i)}(s))]|_{s=u-T_i}\right)\right|\,du>\varepsilon/(4C)\right)\\
        &\leq \frac{4C}{\varepsilon}n^{-\frac{\lambda_0}{2\lambda_1}}\mathbb{E}\left[\int_{\tau_n''}^{\tau_n''+M_n}\left|\sum_{T_i\leq u}\left(\mathcal{A}f(Z_1^{(i)}(u-T_i))-\mathbb{E}[\mathcal{A}f(Z_1^{(i)}(s))]|_{s=u-T_i}\right)\right|\cdot 1\,du\right]\\
        &\leq \frac{4C}{\varepsilon}n^{-\frac{\lambda_0}{2\lambda_1}}M_n\sqrt{\int_{0}^{t_{n,+}^A+M_n}\mathbb{E}\left[\left|\sum_{T_i\leq u}\left(\mathcal{A}f(Z_1^{(i)}(u-T_i))-\mathbb{E}[\mathcal{A}f(Z_1^{(i)}(s))]|_{s=u-T_i}\right)\right|^2\right]\,du}
    \end{align*}
    and similarly to as before, the cross terms for the sum inside the expectation vanish. The expectation also grows slower than $C_2e^{\lambda_0 u}$ so that:
    \begin{align*}
        \leq C_3n^{-\frac{\lambda_0}{2\lambda_1}}M_nn^{\frac{\lambda_0}{2\lambda_1}}\overset{n\to\infty}{\to}0.
    \end{align*}
    In total then, the $|X_{\tau_n,j}^{(1)}-X_{\tau_n',j}^{(1)}|$ tends to $0$ in probability.

    Now, consider the (simpler) case $i=2$. Again, we may write:
    \begin{align*}
        X_{t,j}^{(2)}-X_{s,j}^{(2)}&=\left(e^{-\frac{\lambda_0}{2}t}-e^{-\frac{\lambda_0}{2}s}\right)\int_0^t P_j(t-u)\,dR_u\\
        &+e^{-\frac{\lambda_0}{2}s}\left(\int_0^t P_j(t-u)\,dR_u-\int_0^s P_j(s-u)\,dR_u\right).
    \end{align*}
    Thus:
    \begin{align}
        &\mathbb{P}(\sup_{s\in [0,M_n]}|X_{\tau_n+s,j}^{(2)}-X_{\tau_n,j}^{(2)}|>\varepsilon,\Omega_0^\infty)-\delta\nonumber\\
        &\leq \mathbb{P}\left(e^{-\frac{\lambda_0}{2}t_{n,-}^A}M_n\sup_{s\in [0,M_n]}\left|\int_0^{\tau_n''+s}P_j(\tau_n''+s-u)\,dR_u\right|>\varepsilon/4\right)\label{eq:term_2_part_1}\\
        &+\mathbb{P}(e^{-\frac{\lambda_0}{2}t_{n,-}^A}\sup_{s\in [0,M_n]}\left|\int_0^{\tau_n''}P_j(\tau_n''+s-u)-P_j(\tau_n''-u)\,dR_u\right|>\varepsilon/4)\label{eq:term_2_part_2}\\
        &+\mathbb{P}(e^{-\frac{\lambda_0}{2}t_{n,-}^A}\sup_{s\in [0,M_n]}\left|\int_{\tau_n''}^{\tau_n''+s}P_j(\tau_n''+s-u)\,dR_u\right|>\varepsilon/4).\label{eq:term_2_part_3}
    \end{align}
    Terms \ref{eq:term_2_part_1} and \ref{eq:term_2_part_3} are controlled in precisely the same manner. Thus, we just control \ref{eq:term_2_part_1}. First, using integration by parts (if $j\neq 1$)
    \begin{align*}
        \int_0^{\tau_n''+s}P_j(\tau_n''+s-u)\,dR_u=\int_0^{\tau_n''+s}P_j'(\tau_n''+s-u)R_u\,du.
    \end{align*}
    Note that $\sup_{t>0}P_j'(t)=C_j<\infty$ so that using Markov's inequality and Cauchy-Schwarz:
    \begin{align*}
        \ref{eq:term_2_part_1}\leq \frac{4C}{\varepsilon}n^{-\frac{\lambda_0}{2\lambda_1}}M_n\int_0^{\frac{1}{\lambda_1}\ln n+A+M_n}\sqrt{\mathbb{E}[|R_u|^2]}\,du\overset{n\to\infty}{\to}0
    \end{align*}
    since $\mathbb{E}[|R_u|^2]=\mathcal{O}(e^{\lambda_0 u})$. If $j=1$, we obtain an additional term of magnitude $R_{\tau_n''+s}$, so that to control term \ref{eq:term_2_part_1} we must show:
    \begin{align}
        \mathbb{P}\left(e^{-\frac{\lambda_0}{2}t_{n,-}^A}M_n\sup_{s\in [0,M_n]}\left|R_{\tau_n''+s}\right|>\varepsilon/8\right)\overset{n\to\infty}{\to}0.\label{eq:term_2_part_1_special_j}
    \end{align}
    By the Optional Stopping theorem, $R_{\tau_n''+\cdot}$ is a $\mathcal{I}_{\tau_n''+\cdot}$ martingale, whose quadratic variation at $s=M_n$ is $[R_{\tau_n''+\cdot}]_{M_n}=\#\{T_i\in [\tau_n'',\tau_n''+M_n]\}$. By using the series of inequalities as in \ref{eq:expectation_upper_1_control_2_1} (except with $E_j^{(i)}$ replaced by $1$), we obtain:
    \begin{align*}
        \mathbb{E}\left[[R_{\tau_n''+\cdot}]_{M_n}\right]\leq CM_n e^{\lambda_0 t_{n,+}^A}.
    \end{align*}
    Then by an application of Doob's maximal inequality, and Cauchy-Schwarz, \ref{eq:term_2_part_1_special_j} holds (precisely analogous to how \ref{eq:1_control_2_1} was controlled).
    
    For term \ref{eq:term_2_part_2} we use the fundamental theorem of calculus and Fubini to see:
    \begin{align*}
        \int_0^{\tau_n''}P_j(\tau_n''+s-u)-P_j(\tau_n''-u)\,dR_u=\int_0^s\int_0^{\tau_n''}P_j'(\tau_n''+z-u)\,dR_u\,dz.
    \end{align*}
    Therefore:
    \begin{align*}
        \ref{eq:term_2_part_2}\leq \mathbb{P}\left(e^{-\frac{\lambda_0}{2} t_{n,-}^A}\int_0^{M_n}\left|\int_0^{\tau_n''}P_j'(\tau_n''+z-u)\,dR_u\right|\,dz>\varepsilon/4\right)
    \end{align*}
    and by integration by parts (noting the uniform bound $|P_j''|\leq C$):
    \begin{align*}
        &\leq  \mathbb{P}\left(e^{-\frac{\lambda_0}{2} t_{n,-}^A}\int_0^{M_n}\left[|P_j'(z)|\sup_{t\in [0,t_{n,+}^A]}(|R_{t}|)+\int_0^{\tau_n''}C|R_u|\,du\right]dz>\varepsilon/4\right)\\
        &\leq Cn^{-\frac{\lambda_0}{2\lambda_1}}\int_0^{M_n}\left[\sqrt{\mathbb{E}\left[\sup_{t\in [0,t_{n,+}^A]}|R_t|^2\right]}+\int_0^{t_{n,+}^A}\sqrt{\mathbb{E}\left[|R_u|^2\right]}\,du\right]dz\\
        &\leq Cn^{-\frac{\lambda_0}{2\lambda_1}}\int_0^{M_n}\left[2\sqrt{\mathbb{E}[|R_{t_{n,+}^A}|^2]}+\int_0^{t_{n,+}^A}\sqrt{\mathbb{E}[|R_u|^2]}\,du\right]dz\to 0.
    \end{align*}
    Hence, $|X_{\tau_n,j}^{(2)}-X_{\tau_n',j}^{(2)}|$ converges to $0$ in probability. 
    
    Finally, we examine the $i=3$ term:
\begin{align}
        &\mathbb{P}(\sup_{s\in [0,M_n]}|X_{\tau_n+s,j}^{(3)}-X_{\tau_n,j}^{(3)}|>\varepsilon,\Omega_0^\infty)-\delta\nonumber\\
        &\leq \mathbb{P}\left(e^{-\frac{\lambda_0}{2}t_{n,-}^A}M_n\sup_{s\in [0,M_n]}\left|\int_0^{\tau_n''+s}P_{j}(t-s)(Z_0(u)-e^{\lambda_0 u}Y)\,du\right|>\varepsilon/4\right)\label{eq:term_3_part_1}\\
        &+\mathbb{P}(e^{-\frac{\lambda_0}{2}t_{n,-}^A}\sup_{s\in [0,M_n]}\left|\int_{\tau_n''}^{\tau_n''+s}P_j(t-s)(Z_0(u)-e^{\lambda_0 u}Y)\,du\right|>\varepsilon/4).\label{eq:term_3_part_2}
    \end{align}
Since $\mathbb{E}\left[|Z_0(u)-e^{\lambda_0 u}Y|^2\right]=O(e^{\lambda_0 u})$ we can control term \ref{eq:term_3_part_1}:
    \begin{align*}
        \ref{eq:term_3_part_1}\leq C n^{-\frac{\lambda_0}{2\lambda_1}}M_n\int_0^{t_{n,+}^A+M_n}\sqrt{\mathbb{E}\left[|Z_0(u)-e^{\lambda_0 u}Y|^2\right]}\,du=\mathcal{O}(M_n).
    \end{align*}
By the strong Markov property, we control term \ref{eq:term_3_part_2}:
    \begin{align*}
        \ref{eq:term_3_part_2}&\leq Cn^{-\frac{\lambda_0}{2\lambda_1}}\int_0^{M_n}\sqrt{\mathbb{E}\left[|Z_0(u+\tau_n'')-e^{\lambda_0 (u+\tau_n'')}Y|^2\right]}\,du\\
        &\leq Cn^{-\frac{\lambda_0}{2\lambda_1}}\int_0^{M_n}\sqrt{\mathbb{E}\left[\mathbb{E}\left[|Z_0(u)-e^{\lambda_0 u}Y|^2|Z_0(\tau_n'')\right]\right]}\,du\\
        &\leq Cn^{-\frac{\lambda_0}{2\lambda_1}}\int_0^{M_n}\sqrt{\mathbb{E}\left[Z_0(\tau_n'')e^{\lambda_0 u}\right]}\,du\\
        &\leq Cn^{-\frac{\lambda_0}{2\lambda_1}}\int_0^{M_n}\sqrt{\mathbb{E}\left[\sup_{t\in [0,\infty)}(Z_0(t)e^{-\lambda_0 t})\right]}e^{\frac{\lambda_0}{2}(t_{n,+}^A+u)}\,du\overset{n\to\infty}=\mathcal{O}(M_n).
    \end{align*}
\end{proof}

We now can obtain convergence of $X_{\tau_n',j}$ by following the steps laid out in Proposition \ref{prop:weak_topology_convergence}.

\begin{proposition}\label{prop:gamma_1_tau_n_prime}
    Let $M_n=n^{-\gamma}$ for $\gamma\in (0,(\lambda_1-\lambda_0)/(4\lambda_1))$. If we take $t_n=\frac{2}{\lambda_1-\lambda_0}\ln (M_n^{-1})$, then:
    \begin{enumerate}
        \item $\tau_n'-t_n\to\infty$ almost surely on $\Omega_0^\infty$
        \item $e^{-\lambda_0 \tau_n'/2}\max\{S_j^{[0,t_n]}(\tau_n'),Y\int_0^{t_n}P_j(\tau_n'-s)e^{\lambda_0 s}\,ds\}$ converges to $0$ in probability, on $\Omega_0^\infty$.
    \end{enumerate}
    Using such $\gamma$ and $t_n$:
    \begin{align*}
        \{X_{\tau_n',\geq 1}\}\cup(X_{\tau_n',j})_{j=1}^\infty|\Omega_0^\infty\Rightarrow (X_{\infty,j})_{j=1}^\infty
    \end{align*}
    $\mathcal{I}_\infty$-stably.
\end{proposition}
\begin{proof}

    We first prove that $\gamma$ satisfying those above two conditions can be chosen. Note that $\tau_n'\geq \tau_n\sim \frac{1}{\lambda_1}\ln(n)$ almost surely on $\Omega_0^\infty$. Thus, condition (1) is immediate.

    For condition (2), fix any $\varepsilon>0$ and it suffices to show:
    \begin{align}
        &\mathbb{P}(e^{-\lambda_0\tau_n'/2}S_j^{[0,t_n]}(\tau_n')>\varepsilon,\Omega_0^\infty)\to 0\label{eq:tau_n_prime_term_1}\\
        &\mathbb{P}(e^{-\lambda_0\tau_n'/2}Y\int_0^{t_n}P_j(\tau_n'-s)e^{\lambda_0 s}\,ds>\varepsilon,\Omega_0^\infty)\to 0\label{eq:tau_n_prime_term_2}.
    \end{align}

    For term \ref{eq:tau_n_prime_term_1}, we can further be on the set $\{\tau_n'>t_n\}$ (since this is achieved with arbitrarily high probability on $\Omega_0^\infty$ as $n\to\infty$ by condition (1)). Note that $S_j^{[0,t_n]}(\tau_n')\leq M_{t_n}$. Therefore:
    \begin{align*}
        \ref{eq:tau_n_prime_term_1}-\delta\leq \mathbb{P}(e^{-\lambda_0 t_{n,-}^A/2}M_{t_n}>\varepsilon)\leq C n^{-\frac{\lambda_0}{2\lambda_1}}\mathbb{E}[M_{t_n}]=Cn^{-\frac{\lambda_0}{2\lambda_1}+\frac{2\lambda_0\gamma}{\lambda_1-\lambda_0}}\overset{n\to\infty}{\to}0,
    \end{align*}
    for the specified $\gamma$ satisfying (1). Similarly, for such small $\gamma$, \ref{eq:tau_n_prime_term_2} tends to $0$, as desired.

    The proof here follows that of Proposition \ref{prop:weak_topology_convergence} but with small adaptations. In that proof, we fixed $T$, took $t\to\infty$ and then took $T\to\infty$. Here, we consider the sequence $t_n$ instead of $T$ and take $n\to\infty$ so that $\tau_n'\to\infty$ at the same time as $t_n\to\infty$. $\gamma$ is chosen small enough so that $t_n$ grows to $\infty$ at a slow rate. Using the same notation as in the proof of that proposition, we may decompose:
    \begin{align*}
    X_{\tau_n',j}=e^{-\lambda_0 t_n/2}\sum_{i=1}^{Z_0(t_n)}X^{t_n,(i)}_{\tau_n'-t_n,j}+e^{-\lambda_0 \tau_n'/2}S_j^{[0,t_n]}(t)+e^{-\lambda_0 \tau_n'/2}Y\int_0^{t_n} P_j(\tau_n'-s)e^{\lambda_0 s}\,ds,
\end{align*}
By the conditions on $\gamma$, the last two remainder terms tend to $0$ in probability on $\Omega_0^\infty$. At this stage in Proposition \ref{prop:weak_topology_convergence}, we used the $\sigma$-algebra $\mathcal{I}_{t_n}$. Here, we will instead consider the $\sigma$-algebra $\mathcal{I}^\ast_{t_n}$, which is the smallest one containing (1) $\mathcal{I}_{t_n}$ and (2) the history from inception to $\infty$ of all type-1 clones produced during the interval $[0,t_n]$. Note that $\tau_n'$ is $\mathcal{I}^\ast_{t_n}$-measurable. Therefore:
\begin{align*}
    \mathbb{E}[e^{ie^{-\lambda_0 t_n/2}\sum_{i=1}^{Z_0(t_n)}\langle \boldsymbol{\theta},\boldsymbol{X}_{\tau_n'-t_n}^{t_n,(i)}\rangle}|\mathcal{I}^\ast_{t_n}]&=\mathbb{E}[e^{i e^{-\lambda_0 t_n/2} \langle \boldsymbol{\theta}, \boldsymbol{X}_{s-t_n}^{t_n,(1)}\rangle}]|_{s=\tau_n'}^{Z_0(t_n)}\nonumber\\
    &=\left(1-\frac{1}{2}e^{-\lambda_0 T}\boldsymbol{\theta}^\intercal\Sigma_{\tau_n',t_n}\boldsymbol{\theta}+o_{n\to\infty}(e^{-\lambda_0 t_n})\right)^{Z_0(t_n)},
\end{align*}
which on the event $\Omega_0^\infty$, and using the two conditions for the chosen $\gamma$, tends to $e^{-\frac{1}{2}Y\boldsymbol{\theta}^\intercal \Sigma\boldsymbol{\theta}}$ as $n\to\infty$. The rest of the proof follows that of Proposition \ref{prop:weak_topology_convergence}, with obvious adaptations. Once again, we also get stable convergence by the remark after Proposition \ref{prop:weak_topology_convergence} and by using that $\mathcal{I}_t\subseteq \mathcal{I}^\ast_t\subseteq \mathcal{I}_\infty$ (so $\mathcal{I}^\ast_\infty=\mathcal{I}_\infty$).
\end{proof}

We now begin to prove the main theorem.

\CLTDetectionSize*

We first show convergence under the product topology.
\begin{proposition}
    Theorem \ref{thm:clt_detection_size} holds, where weak convergence is taken under the product topology. 
\end{proposition}
\begin{proof}
    Let $\gamma>0$ be as in Proposition \ref{prop:gamma_1_tau_n_prime}. Then:
    \[(X_{\tau_n,j})_{j=1}^\infty|\Omega_0^\infty=(X_{\tau_n',j})_{j=1}^\infty+(X_{\tau_n,j}-X_{\tau_n',j})_{j=1}^\infty|\Omega_0^\infty\overset{n\to\infty}{\Rightarrow}(X_{\infty,j})_{j=1}^\infty|\Omega_0^\infty\]
    $\mathcal{I}_\infty$-stably, using Lemma \ref{lem:tau_n_tau_n_prime_X} and Proposition \ref{prop:gamma_1_tau_n_prime}. Writing $\tau_n=\frac{1}{\lambda_1}\ln(n/W)+\left(\tau_n-\frac{1}{\lambda_1}\ln(n/W)\right)$, we then obtain:
    \begin{align*}
        &(n/W)^{-\frac{\lambda_0}{2\lambda_1}}\left(S_j(\tau_n)-(av+\mu)Y\int_0^{\tau_n}P_j(\tau_n-s)e^{\lambda_0 s}\,ds\right)_{j=1}^\infty|\Omega_0^\infty\\
        &=e^{\frac{\lambda_0}{2}(\tau_n-\frac{1}{\lambda_1}\ln(n/W))}(X_{\tau_n,j})_{j=1}^\infty|\Omega_0^\infty\overset{n\to\infty}{\Rightarrow}(X_{\infty,j})_{j=1}^\infty|\Omega_0^\infty
    \end{align*}
    $\mathcal{I}_\infty$-stably, using that $|\tau_n-\frac{1}{\lambda_1}\ln(n/W)|\to 0$ almost surely on $\Omega_0^\infty$ as $n\to\infty$. Finally, note that $W$ is $\mathcal{I}_\infty$-measurable so that:
    \begin{align*}
        &W^{-\frac{\lambda_0}{2\lambda_1}}(n/W)^{-\frac{\lambda_0}{2\lambda_1}}\left(S_j(\tau_n)-(av+\mu)Y\int_0^{\tau_n}P_j(\tau_n-s)e^{\lambda_0 s}\,ds\right)_{j=1}^\infty|\Omega_0^\infty\overset{n\to\infty}{\Rightarrow}W^{-\frac{\lambda_0}{2\lambda_1}}(X_{\infty,j})_{j=1}^\infty|\Omega_0^\infty,
    \end{align*}
    as desired.
\end{proof}

We may also extend this convergence to be in $\ell^p$ for all $p$ sufficiently large. First, we note a result that is an extension of Lemma 18 of \cite{ahmed2026}

\begin{lemma}\label{lem:V_finite_moments}
    For $q>1$ fixed, define:
    \[V:=\sum_{i}e^{-q\lambda_0 T_i}.\]
    For all $k$, $\mathbb{E}[V^k]<\infty$.
\end{lemma}
\begin{proof}
    The exact same proof as in Lemma 18 of \cite{ahmed2026} works but replacing $2$ in the exponent with $q$.
\end{proof}

We first need an upper bound on the summation of  $p$th moment of $\Gamma_{1,j}(t)$ over $j$, where $\Gamma_{1,j}$ is the same as $\Gamma_{1,j}$ in \ref{eq:X1_decompose_gamma} (with the $j$-dependence made explicit).
\begin{lemma}\label{lemma:QuadVarBound}
For $p$ sufficiently large:
\begin{align}
    &\sum_{j=1}^\infty \EE\left[|\Gamma_{1,j}(t)|^p\right]\lesssim e^{(p/2)\lambda_0t},\label{eq:g1_bound_t}\\
    &\sum_{j=1}^\infty \EE\left[\sup_{t'\leq t}|\Gamma_{2,j}(t')|^p\right]\lesssim e^{(p/2)\lambda_0t}.\label{eq:g2_bound_t}
\end{align}

In fact for $p$ sufficiently large:
\begin{align}
    &\sup_{t\geq 0}e^{-(p/2)\lambda_0t}\sum_{j=R}^\infty \EE\left[|\Gamma_{1,j}(t)|^p\right]\llas{R}{\infty}1,\label{eq:g1_bound_R}\\
    &\sup_{t\geq 0}e^{-(p/2)\lambda_0t}\sum_{j=R}^\infty \EE\left[\sup_{t'\leq t}|\Gamma_{2,j}(t)|^p\right]\llas{R}{\infty}1.\label{eq:g2_bound_R}
\end{align}
\end{lemma}
\begin{proof}
    See Appendix Section \ref{sec:quad_var_bound}
\end{proof}

\begin{proposition}
    Theorem \ref{thm:clt_detection_size} holds, where weak convergence is taken under the $\ell_p$-topology, for $p$ sufficiently large.
\end{proposition}
\begin{proof}
    It sufices to show that for any $\varepsilon>0$, there exists sequence $(r_k)_k$ and positive $M$ such that:
    \begin{align*}
        \mathbb{P}(X_{\tau_n}\not\in K_{(r_k)_k,M},\Omega_0^\infty)\leq \varepsilon.
    \end{align*}
    for all $n\geq 1$. Here:
    \[K_{(r_k)_k,M}:=\{x\in \ell^p:\Vert x\Vert_p\leq M,\forall k\in\mathbb{N},\sum_{j=r_k}^\infty |x_j|^p\leq 1/k\}\]
    To show this, let $\delta=\varepsilon/2$ and observe that there is $A$ large enough such that:
    \begin{align*}
        \mathbb{P}(X_{\tau_n}\not\in K_{(r_k)_k,M},\Omega_0^\infty)-\delta\leq \mathbb{P}(X_{\tau_n}\not\in K_{(r_k)_k,M},\tau_n\in I_n^A).
    \end{align*}
    We may further decompose
    \begin{align}
        \mathbb{P}(X_{\tau_n}\not\in K_{(r_k)_k,M},\tau_n\in I_n^A)&\leq \mathbb{P}(\sup_{t\in I_n^A}\Vert X_t\Vert_p\geq M)\label{eq:tight_taun_term1}\\
        &+\sum_{k=1}^\infty \mathbb{P}\left(\sum_{j=r_k}^\infty |X_{\tau_n,j}|^p\geq 1/k,\tau_n\in I_n^A\right)\label{eq:tight_taun_term2}
    \end{align}
    First we control \ref{eq:tight_taun_term1}. To do so, we use the decomposition of $X_t$ considered in \ref{eq:X_t_decomposition} to get:
    \begin{align}
        \ref{eq:tight_taun_term1}&\leq \mathbb{P}(\sup_{t\in I_n^A}\Vert X_t^{(1)}\Vert_p\geq M/3)\label{eq:tight_taun_term1_1}\\
        &+ \mathbb{P}(\sup_{t\in I_n^A}\Vert X_t^{(2)}\Vert_1\geq M/3)\label{eq:tight_taun_term1_2}\\
        &+ \mathbb{P}(\sup_{t\in I_n^A}\Vert X_t^{(3)}\Vert_1\geq M/3)\label{eq:tight_taun_term1_3}
    \end{align}
    To control \ref{eq:tight_taun_term1_1}, we use the decomposition \ref{eq:X1_decompose_gamma}:
    \begin{align*}
        \ref{eq:tight_taun_term1_1}&\leq \mathbb{P}(e^{-\frac{p}{2}\lambda_0t_{n,-}^A}\sup_{t\in I_n^A}\sum_{j=1}^\infty |\Gamma_{1,j}(t)|^p\geq M^p/2^p) +\mathbb{P}(e^{-\frac{\lambda_0}{2}t_{n,-}^A}\sup_{t\in I_n^A}\sum_{j=1}^\infty |\Gamma_{2,j}(t)|^p\geq M^p/2^p).
    \end{align*}
    $\sum_{j=1}^\infty|\Gamma_{1,j}(t)|^p$ is a $\mathcal{I}_t$-submartingale, so that by Doob's submartingale inequality and Lemma \ref{lemma:QuadVarBound} (specifically, bound \ref{eq:g1_bound_t}):
    \begin{align}
        &\mathbb{P}(e^{-\frac{p}{2}\lambda_0t_{n,-}^A}\sup_{t\in I_n^A}\sum_{j=1}^\infty |\Gamma_{1,j}(t)|^p\geq M^p/2^p)\leq \frac{C}{M^p}n^{-\frac{p}{2}\frac{\lambda_0}{\lambda_1}}\sum_{j=1}^\infty\mathbb{E}[|\Gamma_{1,j}(t_{n,+}^A)|^p]\\
        &\lesssim 1/M^p\nonumber
    \end{align}
    uniformly in time.

    For the $\Gamma_{2,j}$ term, we use Markov's inequality and Lemma \ref{lemma:QuadVarBound} (bound \ref{eq:g2_bound_t}). Now for \ref{eq:tight_taun_term1_2}, using integration by parts:
    \begin{align*}
        X_{t,j}^{(2)}=\int_0^t P_{j}'(t-s)R_s\,ds.
    \end{align*}
    Notice that for all $s\geq 0$, $\sum_{j=1}^\infty |P_j'(s)|\leq C<\infty$. Therefore:
    \begin{align*}
        \ref{eq:tight_taun_term1_2}&\leq \mathbb{P}(e^{-\frac{\lambda_0}{2}t_{n,-}^A}C\sup_{t\in I_n^A}\int_0^t |R_s|\,ds\geq M)\\
        &\leq \frac{1}{M}C_2n^{-\frac{\lambda_0}{2\lambda_1}}\int_0^{t_{n,+}^A}\sqrt{\mathbb{E}[|R_s|^2]}\,ds\\
        &\leq \frac{1}{M}C_3n^{-\frac{\lambda_0}{2\lambda_1}}n^{\frac{\lambda_0}{2\lambda_1}}=\frac{C_3}{M}\overset{M\to\infty}{\to}0.
    \end{align*}
    For \ref{eq:tight_taun_term1_3}, we may quite similarly control it using that $\sum_{j=1}^\infty P_j(s)\leq C$:
    \begin{align*}
        \ref{eq:tight_taun_term1_3}&\leq \mathbb{P}(e^{-\frac{\lambda_0}{2}t_{n,-}^A}C\sup_{t\in I_n^A} \int_0^t |Z_0(s)-Ye^{\lambda_0 s}|\,ds\geq M)\\
        &\leq \frac{1}{M}n^{-\frac{\lambda_02}{2\lambda_1}}C_2\int_0^{t_{n,+}^A}\sqrt{\mathbb{E}[|Z_0(s)-Ye^{\lambda_0s}|^2]}\,ds\\
        &\leq \frac{1}{M}C_3\overset{M\to\infty}{\to}0.
    \end{align*}

    Finally, we control \ref{eq:tight_taun_term2}
    \begin{align}
        \ref{eq:tight_taun_term2}&\leq \sum_{k=1}^\infty\mathbb{P}(\sup_{t\in I_n^A}\Vert (X_{t,j}^{(1)})_{j=r_k}^\infty\Vert_p\geq 1/k^{1/p})\label{eq:tight_taun_term2_1}\\
        &+ \sum_{k=1}^\infty\mathbb{P}(\sup_{t\in I_n^A}\Vert (X_{t,j}^{(2)})_{j=r_k}^\infty\Vert_1\geq 1/k^{1/p})\label{eq:tight_taun_term2_2}\\
        &+ \sum_{k=1}^\infty\mathbb{P}(\sup_{t\in I_n^A}\Vert (X_{t,j}^{(3)})_{j=r_k}^\infty\Vert_1\geq 1/k^{1/p})\label{eq:tight_taun_term2_3}.
    \end{align}
    These terms are dealt with in an analogous manner as terms \ref{eq:tight_taun_term1_1}, \ref{eq:tight_taun_term1_2}, \ref{eq:tight_taun_term1_3}, but replacing bounds \ref{eq:g1_bound_t} and \ref{eq:g2_bound_t} with bounds \ref{eq:g1_bound_R} and \ref{eq:g2_bound_R}, respectively.
\end{proof}

\subsection{Asymptotic variance for the class of estimators $\{\widehat{b}_{\leq j}(t)\}_{j\in\mathbb{N}}$}\label{app:asymp_var_estimators}

Here we prove Proposition \ref{prop:asymp_var_estimators}.

\begin{proof}
    Remark that $F_{\leq j}(\ell)$ is holomorphic over the region $\ell\in \{x+iy:x>0\}:=C_{>0}$. Furthermore, $F_{\leq j}(\ell)$ converges to $1$ as $j\to\infty$ uniformly over small disks $D\subseteq C_{>0}$ centered about any point in $C_{>0}$. Hence, derivatives converge to that of the limiting function:
\begin{align*}
    \frac{d}{d\ell}\sum_{k=1}^j F_k(\ell)|_{\ell=b}\overset{j\to\infty}{\to}\frac{d}{d\ell}[1]|_{\ell=b}=0.
\end{align*}

As $\sigma_{i_1,i_2}$ is absolutely summable and $\mu_{\leq j}\to \mu_{\geq 1}$ as $j\to \infty$, the numerator also tends to $0$. We show that the denominator converges to $0$ quickly relative to the numerator, in which case $\sigma_{b,\leq j}^2\to \infty$ as $j\to\infty$. First we investigate the numerator and the rate at which each term approaches its limit as $j\to\infty$. Note:
\begin{align*}
    \mu_{\geq1 }-\mu_{\leq j}=\mu_{\geq j+1}=\Theta_{j\to\infty}\left(j^{-\lambda_0/\lambda_1}\right).
\end{align*}

Write $\delta_j=\mu_{>j}/\mu_{\geq 1}$ and $S_{J,I}=\sum_{i_1\in J,i_2\in I}\sigma_{i_1,i_2}$, where $J,I\in \mathcal{P}(\mathbb{N})$. The numerator of $\sigma_{b,\leq j}^2$ can be written as:
\begin{align*}
    &S_{\leq j,\leq j}-2(1-\delta_j)S_{\leq j,<\infty}+(1-\delta_j)^2S_{<\infty,<\infty}\\
    &=S_{\leq j,\leq j}\delta_j^2-2(1-\delta_j)(S_{\leq j,\infty}-S_{\leq j,\leq j})+(1-\delta_j)^2(S_{<\infty,<\infty}-S_{\leq j,\leq j}).
\end{align*}

Writing $S_{\leq j,<\infty}-S_{\leq j,\leq j}=S_{\leq j,>j}$ and $S_{<\infty,<\infty}-S_{\leq j,\leq j}=2S_{\leq j,>j}+S_{>j,>j}$, we obtain:
\begin{align*}
    &=S_{\leq j,\leq j}\delta_j^2-2S_{\leq j,>j}(1-\delta_j)\delta_j+(1-\delta_j)^2 S_{>j,>j}.
\end{align*}

As $\sum_{i_1,i_2=1}^\infty \sigma_{i_1,i_2}$ converges absolutely, $S_{\leq j,\leq j}\to S_{<\infty,\infty}$ as $j\to\infty$ and $S_{\leq j,>j}=o_{j\to \infty}(1)$. Therefore:
\begin{align*}
    S_{\leq j,\leq j}\delta_j^2=o_{j\to\infty}(j^{-\lambda_0/\lambda_1})\text{ and }S_{\leq j,>j}\delta_j=o_{j\to\infty}(j^{-\lambda_0/\lambda_1}).
\end{align*}

The third term requires more delicate control of $S_{>j,>j}$. Note by the asymptotics of $\sigma_{j,k}$ and integral comparisons, there exists constants $C_1,C_2>0$ such that for $j$ large enough:
\[C_1\int_{j}^\infty\int_j^\infty \frac{1}{xy(x+y)^{\lambda_0/\lambda_1}}\,dy\,dx+C_2\int_j^\infty\frac{1}{x^{1+\lambda_0/\lambda_1}}\,dx-S_{>j,>j}=\mathcal{O}_{j\to\infty}(j^{-1-\lambda_0/\lambda_1}).\]
So, if we show that:
\begin{align}
    &\int_j^\infty \int_j^\infty\frac{1}{xy(x+y)^{\lambda_0/\lambda_1}}\,dy\,dx=\Theta(j^{-\lambda_0/\lambda_1}), \text{ and}\label{eq:asymp_integral_1}\\
    &\int_{j}^\infty \frac{1}{x^{1+\lambda_0/\lambda_1}}=\Theta(j^{-\lambda_0/\lambda_1}),\label{eq:asymp_integral_2}
\end{align}
then so to will $S_{>j,>j}=\Theta(j^{-\lambda_0/\lambda_1})$. If this is so, the numerator of the variance is $\Theta(j^{-\lambda_0/\lambda_1})$. \ref{eq:asymp_integral_2} is trivial by computing the integral. Now, by using the change of variables $y=xt$ :
\begin{align*}
    \ref{eq:asymp_integral_1}=\int_{j}^\infty x^{-1-\lambda_0/\lambda_1}\int_{j/x}^\infty \frac{1}{t(1+t)^{\lambda_0/\lambda_1}}\,dt\,dx.
\end{align*}
Splitting the range of integration for the inner integral between $j/x$ to $1$ and $1$ to $\infty$ yields:
\begin{align*}
    =\Theta_{j\to\infty}\left(\int_j^\infty x^{-1-\lambda_0/\lambda_1}(\log(x/j)+C)\,dx\right)=\Theta_{j\to\infty}(j^{-\lambda_0/\lambda_1}),
\end{align*}
as desired.

For the denominator, we exploit the holomorphicity of $F_{\leq j}(\ell)$. Fix $b\in\mathbb{R}_{>0}$. Fix any small $\varepsilon>0$. By extending the proof of Proposition 1 of \cite{ahmed2026}, there exists a disk in $C_{>0}$, $D_{\varepsilon}$,  around $b$ sufficiently small so that:
\[\mu_{\geq j}(\ell)=(av+\mu)\frac{q_1(\ell)\Gamma(\lambda_0/\lambda_1(\ell)+1)}{\lambda_0j^{\lambda_0/\lambda_1(\ell)}}(1-p_1(\ell))^{-\lambda_0/\lambda_1(\ell)}+\mathcal{O}_{j\to\infty}(j^{-1-\lambda_0/\lambda_1+\varepsilon}),\]
with the error $\mathcal{O}$ term uniform for $\ell\in D_{\varepsilon}$. Denote the error term $R_j(\ell)$. Notice that using the Cauchy integral formula inequalities $|R_j'(b)|\leq C\max_{\gamma\in D_{\varepsilon}}|R_j(\gamma)|=\mathcal{O}_{j\to\infty}(j^{-1-\lambda_0/\lambda_1+\varepsilon})$. So therefore, taking derivatives:
\[\mu_{\geq j}'(b)=\Theta(j^{-\lambda_0/\lambda_1}\log j),\]
implying:
\[F'_{\geq j}(b)=\Theta(j^{-\lambda_0/\lambda_1}\log j).\]
So, we obtain:
\begin{align*}
    \sigma_{b,\leq j}^2=\Theta(j^{\lambda_0/\lambda_1}(\log j)^{-2}).
\end{align*}

\end{proof}

\printbibliography

\appendix

\section{Appendix}

\subsection{Simulation Details}

A modification of the Gillespie algorithm is used to simulate traces of the birth-death process. The modification enables tracking the ancestral relationships of the whole population. The default parameter values used are described in Table \ref{table:default_params}. Note these are the same as in \cite{ahmed2026}.

\begin{table}[H]
  \centering
  \begin{tabular}{|c|p{2.5cm}|}
\hline
\textbf{Parameter} & \textbf{Default Value} \\
\hline
$a$ & 0.1\\
$d$ & 0.001\\
$v$ & 0.01\\
$\mu$ & 0\\
$b$ & 0.04\\
\hline
\end{tabular}
\caption{Default simulation parameter values}\label{table:default_params}
\end{table}

For both Figures \ref{fig:b_1_consistency} and \ref{fig:b_j_deviation}, the same underlying simulation data are used. 1000 runs are performed and stopped at detection sizes $n\in \{10^4,10^5,10^6,10^6,4\cdot 10^7\}$ (with runs going extinct being ignored). 

\subsection{Uniform fourth moment control: Proof of \ref{lem:control_fourth_moment}}\label{sec:control_fourth_moment}
\begin{proof}
    We first decompose $X_{t,j}$ into three pieces:
    \begin{align}
        e^{\lambda_0t/2}X_{t,j}&=\sum_{T_i\le t}(\mathbf{1}_{Z_1^{(i)}(t-T_i)=j}-P_j(t-T_i))\label{eq:fourth_moment_term_1}\\
        &+\int_0^t P_j(t-s)\,dR_s\label{eq:fourth_moment_term_2}\\
        &+(av+\mu)\int_0^t P_j(t-s)(Z_0(s)-Ye^{\lambda_0s})\,ds.\label{eq:fourth_moment_term_3}
    \end{align}
    It suffices to show the fourth moment for each is $\mathcal{O}(e^{2\lambda_0 t})$. The first term is controlled by Rosenthal's inequality. Using that $\mathbb{E}[(\mathbf{1}_{Z_1^{(i)}(t-T_i)}-P_j(t-T_i))^p|\mathcal{B}_\infty]=0$, we then obtain:
    \begin{align*}
        \mathbb{E}\left[\left(\sum_{T_i\leq t}(\mathbf{1}_{Z_1^{(i)}(t-T_i)}-P_j(t-T_i))\right)^4\right]&\leq C\mathbb{E}\left[(\sum_{T_i\leq t}1^2)^2+\sum_{T_i\leq t}1^4\right]\\
        &\leq 2\mathbb{E}[M_t^2]=\mathcal{O}_{t\to\infty}(e^{2\lambda_0 t}).
    \end{align*}

    The last two terms are controlled via the BDG inequalities. For term \ref{eq:fourth_moment_term_2}, we first use integration by parts to get:
    \begin{align*}
        \ref{eq:fourth_moment_term_2}=\int_0^t P_j'(t-s)R_s\,ds.
    \end{align*}
    Note that $P_j'$ is uniformly bounded, say by $K$. Therefore, by the Minkowsky integral inequality along with the BDG inequalities (using \ref{eq:R_var}), we obtain:
    \begin{align*}
        &\left(\mathbb{E}\left[\left(\int_0^t P_j'(t-s)R_s\,ds\right)^4\right]\right)^{1/4}\\
        &\leq K\int_0^t\mathbb{E}\left[R_s^4\right]^{1/4}\,ds\leq K\int_0^t\mathbb{E}\left[M_s^2\right]^{1/4}\,ds=\mathcal{O}_{t\to\infty}(e^{\frac{1}{2}\lambda_0 t/2}).
    \end{align*}
    Thus:
    \begin{align*}
        \mathbb{E}\left[\ref{eq:fourth_moment_term_2}^4\right]=\mathcal{O}_{t\to\infty}(e^{2\lambda_0 t}).
    \end{align*}
    Remark that $\mathbb{E}\left[(e^{-\lambda_0 s}Z_0(s)-Y)^4\right]=\mathcal{O}_{s\to\infty}(e^{-2\lambda_0 s})$. Therefore, by the Minkowsky triangle inequality, we immediately get:
    \begin{align*}
        \mathbb{E}\left[\ref{eq:fourth_moment_term_3}^4\right]=\mathcal{O}_{t\to\infty}\left(e^{2\lambda_0 t}\right).
    \end{align*}

    The lemma is then proved.

\end{proof}

\subsection{A Rosenthal-maximal inequality for increasing processes}
Here, we prove a Rosenthal-like inequality that will prove useful:
\begin{lemma}\label{lem:rosenthal_maximal}
    Let $\{V_i(\cdot)\}_{1\leq i\leq N}$ be independent, non-decreasing, non-negative cadlag processes. Fix any $T>0$ and $p\geq 2$. Let $V_i(T)\in L^p$ for $1\leq i\leq N$. Then:
    \begin{align*}
        \mathbb{E}\left[\sup_{t\leq T}\left|\sum_{i=1}^N V_i(t)-\mathbb{E}[V_i(t)]\right|^p\right]\leq C_p\mathbb{E}\left[\sum_{i=1}^N \mathbb{E}[V_i(T)^p]+\left(\sum_{i=1}^N \mathbb{E}[V_i(T)^2]\right)^{p/2}\right].
    \end{align*}
\end{lemma}
\begin{proof}
    Insert further independent processes $\{V_i'(\cdot)\}_{1\leq i\leq n}$ where for each $1\leq i\leq n$, $V_i'$ is distributed as $V_i$. We also insert i.i.d. Rademacher r.v's $\{\varepsilon_i\}_{1\leq i\leq n}$ Let $\mathcal{F}$ denote the $\sigma$-algebra generated by  $\{V_i(\cdot)\}_{1\leq i\leq N}$. By Jensen's and symmetrization:
    \begin{align*}
        \mathbb{E}\left[\sup_{t\leq T}\left|\sum_{i=1}^N V_i(t)-\mathbb{E}[V_i(t)]\right|^p\right]&=\mathbb{E}\left[\sup_{t\leq T}\left|\sum_{i=1}^N V_i(t)-\mathbb{E}[V_i'(t)|\mathcal{F}]\right|^p\right]\\
        &\leq \mathbb{E}\left[\sup_{t\leq T}\left|\sum_{i=1}^N V_i(t)-V_i'(t)\right|^p\right]\\
        &\leq 2^p\mathbb{E}\left[\sup_{t\leq T}\left|\sum_{i=1}^N \varepsilon_i V_i(t)\right|^p\right].
    \end{align*}
    Now, add in i.i.d. uniform r.v.'s over $[0,1]$ $\{U_i\}_{1\leq i\leq n}$ and let $\mathcal{E}$ denote the $\sigma$-algebra generated by $\{\varepsilon_i\}_{1\leq i\leq n}$. Then $V_i(t)=\mathbb{E}[V_i(T)\mathbf{1}_{U_iV_i(T)\leq V_i(t)}|\mathcal{E}\otimes \mathcal{F}]$ by the non-decreasing and non-negativity of the $V_i$ processes. Therefore, once again using Jensen's:
    \begin{align*}
        \mathbb{E}\left[\sup_{t\leq T}\left|\sum_{i=1}^N \varepsilon_i V_i(t)\right|^p\right]\leq \mathbb{E}\left[\sup_{t\leq T}\left|\sum_{i=1}^N \varepsilon_i V_i(T)\mathbf{1}_{U_i V_i(T)\leq V_i(t)}\right|^p\right].
    \end{align*}

    Let $\pi:\{1,\dots,N\}\to \{1,\dots,N\}$ be permutation respecting the first $t$ such that $U_iV_i(T)\leq V_i(t)$. In particular $\pi(1)\in \argmin_{i\in \{1,\dots,N\}}\{\min\{t\geq 0:U_iV_i(T)\leq V_i(t)\}\}$ and $\pi(j+1)\in \argmin_{i\in \{1,\dots,N\}\backslash \pi(\{1,\dots,j\})}\{\min\{t\geq 0: U_iV_i(T)\leq V_i(t)\}\}$. Define the $\sigma$-algebra $\mathcal{U}$ generated by $\{U_i\}_{1\leq i\leq n}$ and the $\sigma$-filtration $\mathcal{E}_k=\sigma\{\varepsilon_{\pi(i)}\}_{1\leq i\leq k}$. Note two important things:
    \begin{align*}
        &\sup_{t\leq T}\left|\sum_{i=1}^N \varepsilon_i V_i(T)\mathbf{1}_{U_i V_i(T)\leq V_i(t)}\right|^p=\max_{1\leq k\leq N}\left|\sum_{i=1}^k \varepsilon_{\pi(i)}V_{\pi(i)}(T)\right|^p\\
        &\sum_{i=1}^k\varepsilon_{\pi(i)}V_{\pi(i)}(T)\text{ is a }\mathcal{F}\otimes \mathcal{U}\otimes \mathcal{E}_k\text{ martingale}.
    \end{align*}
    Thus, by Doob's maximal inequality:
    \begin{align*}
        \mathbb{E}\left[\sup_{t\leq T}\left|\sum_{i=1}^N \varepsilon_i V_i(T)\mathbf{1}_{U_i V_i(T)\leq V_i(t)}\right|^p\right]\leq C_p\mathbb{E}\left[\left|\sum_{i=1}^N \varepsilon_i V_i(T)\right|^p\right].
    \end{align*}
    Applying the standard Rosenthal inequality to the last sum gives the desired result.
\end{proof}

\subsection{Proof of Lemma \ref{lemma:QuadVarBound}.}\label{sec:quad_var_bound}

\begin{proof}
We prove the bound \ref{eq:g1_bound_t} first and describe the adaptations needed to show the three other bounds in the end. By the Burkholder-Davis inequality, we get:
\[
\mathbb{E}\left[|\Gamma_{1,j}(t)|^p\right]\leq C_p  \mathbb{E}\left[ \left(\left[\Gamma_{1,j}\right]_{t}\right)^{p/2}\right].
\]
Next note that for $t>0$ we can write,
\[
[\Gamma_{1,j}]_t=\sum_{T_i\leq t} E_{i,j}(t-T_i),
\]
where $E_{i,j}(t)$ is the number of times the $i$th clones enters and exits size $j$ by time $t$. By Lemma 8 in Ahmed et al. \cite{ahmed2026}, $E_{i,j}(\infty)$ is stochastically bounded by a geometric random variable $G$ not depending on $j$. Furthermore, we may bound $E_{i,j}(t)\leq E_{i,j}(\infty)\mathbf{1}_{\zeta_{i,j}\leq s}$, where $\zeta_{i,j}$ is the first time that $Z^{(i)}_1$ reaches size $j$. Combining these factoids, for any $q>0$:
\begin{align}
    \mathbb{E}[E_{i,j}(s)^q]\leq C_q\mathbb{P}(\zeta_{i,j}\leq s)=C_q\mathbb{P}(E_{i,j}(s)\geq 1)\leq C_q\mathbb{E}[E_{i,j}(s)].\label{eq:E_i_j_higher_by_lower}
\end{align}
Denote $\overline{e}_j(s)=\mathbb{E}[E_{i,j}(s)]$. A corollary of \ref{eq:E_i_j_higher_by_lower} and Jensen's inequality is that for $q\geq 1$:
\begin{align}
    \overline{e}_j(s)^q\leq C_q \overline{e}_j(s)\label{eq:E_i_j_higher_by_lower_mean}
\end{align}
We will use these factoids to bound moments and powers of $E_{i,j}$ and $\overline{e}_j$, respectively, throughout the following proof. Trivially:
\begin{align}
    \mathbb{E}[[\Gamma_{1,j}]_t^{p/2}]\leq C_p \left(\mathbb{E}\left[\left|\sum_{T_i\leq t}(E_{i,j}(t-T_i)-\overline{e}_{j}(t-T_i))\right|^{p/2}\right]+\mathbb{E}\left[\left|\sum_{T_i\leq t}\overline{e}_{j}(t-T_i)\right|^{p/2}\right]\right).\label{eq:control_quad_vary}
\end{align}
The first term above can be bounded using Rosenthal's inequality. Since $\{T_i\}$ are $\mathcal{B}_\infty$-measurable and conditional on $\mathcal{B}_\infty$, $\{E_{i,j}(t-T_i);i\geq 1\}$ are independent random variables:
\begin{align}
    &\mathbb{E}\left[\left|\sum_{T_i\leq t}(E_{i,j}(t-T_i)-\overline{e}_{j}(t-T_i))\right|^{p/2}\right]\nonumber\\
    &\leq C\mathbb{E}\left[\sum_{T_i\leq t}\mathbb{E}[|E_{i,j}(s)-\overline{e}_{j}(s)|^{p/2}]|_{s=t-T_i}+\left(\sum_{T_i\leq t}\mathbb{E}[|E_{i,j}(s)-\overline{e}_{j}(s)|^2]|_{s=t-T_i}\right)^{p/4}\right].\label{eq:control_rosenthal_two}
\end{align}
Using \ref{eq:E_i_j_higher_by_lower} and \ref{eq:E_i_j_higher_by_lower_mean}, the first term above is bounded by:
\begin{align*}
    \mathbb{E}\left[\sum_{T_i\leq t}\mathbb{E}[|E_{i,j}(s)-\overline{e}_{j}(s)|^{p/2}]|_{s=t-T_i}\right]\lesssim \int_0^t \overline{e}_j(t-s)e^{\lambda_0 s}\,ds
\end{align*}

uniformly in $j$. Note that $\sum_{j}\overline{e}_j(s)$ is dominated by the mean number of birth and death events that occur within a single type-1 lineage by time $s$, which is $\mathcal{O}_{s\to\infty}(e^{\lambda_1 s})$. Therefore:
\begin{align*}
    \sum_{j=1}^\infty\int_0^t \overline{e}_j(t-s)e^{\lambda_0 s}\,ds\lesssim e^{\lambda_1 t}.
\end{align*}
Next, note that the second terms in \ref{eq:control_quad_vary} and \ref{eq:control_rosenthal_two} are themselves bounded by terms of the form
\begin{align}
    \mathbb{E}\left[\left|\sum_{T_i\leq t}\overline{e}_{j}(t-T_i)\right|^{q}\right],\label{eq:main_control_piece}
\end{align}
where $q=p/4$ or $q=p/2$, respectively. Firstly, by adding and subtracting $\int_0^t \overline{e}_j(t-s)(av+\mu)Z_0(s)\,ds$ and applying Burkholder-Davis:
\begin{align}
    \mathbb{E}\left[\left|\sum_{T_i\leq t}\overline{e}_{j}(t-T_i)\right|^{q}\right]\leq C_q\left(\mathbb{E}\left[\left|\sum_{T_i\leq t}\overline{e}_{j}(t-T_i)\right|^{q/2}\right]+\mathbb{E}\left[\left(\int_0^t \overline{e}_j(t-s)(av+\mu)Z_0(s)\,ds\right)^{q}\right]\right).\label{eq:first_part_control_piece}
\end{align}
In this way, to control \ref{eq:main_control_piece} for $q=p/4$ or $q=p/2$ it suffices to control \ref{eq:main_control_piece} for $q'=p/8$ or $q'=p/4$ and to control the second term in \ref{eq:first_part_control_piece}. Now onto the first term of \ref{eq:first_part_control_piece}, observe that by H\"older's inequality followed by Cauchy-Schwarz:
\begin{align}
    \mathbb{E}\left[\left|\sum_{T_i\leq t}\overline{e}_{j}(t-T_i)\right|^{q'}\right]&\leq C_1\mathbb{E}\left[\left(\sum_{T_i\leq t}\overline{e}_j(t-T_i)^{2q'}e^{2q'\lambda_0 T_i}\right)^{1/2}\left(\sum_{T_i\leq t}e^{-\frac{2q'}{2q'-1}\lambda_0 T_i}\right)^{(2q'-1)/2}\right]\nonumber\\
    &\leq C_1\mathbb{E}\left[\left(\sum_{T_i\leq t}\overline{e}_j(t-T_i)^{2q'}e^{2q'\lambda_0 T_i}\right)\right]^{1/2}\label{eq:former_term}\\
    &\cdot\mathbb{E}\left[\left(\sum_{T_i\leq t}e^{-\frac{2q'}{2q'-1}\lambda_0 T_i}\right)^{2q'-1}\right]^{1/2}.\label{eq:latter_term}
\end{align}
By Lemma \ref{lem:V_finite_moments}, term \ref{eq:latter_term} above is finite and uniformly bounded over all $t\geq 0$. We can use Campbell's Theorem for the former term \ref{eq:former_term}:
\begin{align*}
    \mathbb{E}\left[\left(\sum_{T_i\leq t}\overline{e}_j(t-T_i)^{2q'}e^{2q'\lambda_0 t}\right)\right]&\leq (av+\mu)\int_0^t \overline{e}_j(t-s)^{2q'}e^{(2q'+1)\lambda_0 s}\,ds\\
    &\lesssim te^{(2q'+1)\lambda_0 t}\overline{e}_j(t)^{2q'}.
\end{align*}
with constant uniform in $j$. As long as $q'\geq 1$, we thus have:
\begin{align}
    \sum_{j=1}^\infty \mathbb{E}\left[\left|\sum_{T_i\leq t}\overline{e}_{j}(t-T_i)\right|^{q'}\right]^{1/2}\lesssim \sqrt{t}e^{(q'+1/2)\lambda_0 t}\sum_{j=1}^\infty\overline{e}_j(t)\lesssim \sqrt{t}e^{((q'+1/2)\lambda_0+\lambda_1) t}.\label{eq:e_sum_bound}
\end{align}
We can also bound the second term in \ref{eq:first_part_control_piece} by:
\begin{align*}
    \mathbb{E}\left[\left(\int_0^t \overline{e}_j(t-s)(av+\mu)Z_0(s)\,ds\right)^{q}\right]\leq \left(\int_0^t \overline{e}_j(t-s)e^{\lambda_0 s}\,ds\right)^q(av+\mu)^q\mathbb{E}[\sup_{s\geq 0}(e^{-\lambda_0 s}Z_0(s))^q],
\end{align*}
where the expectation of the supremum is finite by Doob's maximal martingale inequality. Therefore, it remains to control:
\begin{align}
    \int_0^t \overline{e}_j(t-s)e^{\lambda_0 s}\,ds.\label{eq:final_expression_to_control}
\end{align}
The rate at which clones exit/enter size $j$ during time $[v,v+dt]$ is $P_{j-1}(v)(a+b)(j-1)dt+P_{j+1}(v)(j+1)d\cdot dt+(a+b+d)P_j(v)jdt+o_{dt\to 0^+}(dt)$. We can then bound expression \ref{eq:final_expression_to_control} by:
\begin{align*}
&\int_0^t e^{\lambda_0 s}\overline{e}_j(t-s)ds\\
&=
(av+\mu)\int_0^t\int_s^t ((a+b)P_{j-1}(v-s)(j-1)+dP_{j+1}(v-s)(j+1)+(a+b+d)P_j(v-s)j)\,dv e^{\lambda_0 s}\,ds
\end{align*}
Performing the change of variables $r=v-s$:
    \begin{align*}
        &=(av+\mu)\int_0^t\int_0^{t-s} ((a+b)P_{j-1}(r)(j-1)+dP_{j+1}(r)(j+1)+(a+b+d)P_j(r)j)\,dr e^{\lambda_0 s}\,ds\\
        &\leq\frac{(av+\mu)}{\lambda_0}\int_0^t((a+b)P_{j-1}(r)(j-1)+dP_{j+1}(r)(j+1)+(a+b+d)P_j(r)j)e^{\lambda_0(t-r)}\,dr\\
       &=\frac{a+b}{\lambda_0}(j-1)\mathbb{E}[S_{j-1}(t)]+\frac{d}{\lambda_0}(j+1)\mathbb{E}[S_{j+1}(t)]+\frac{a+b+d}{\lambda_0}j\mathbb{E}[S_{j}(t)]\leq Cj^{-\frac{\lambda_0}{\lambda_1}}e^{\lambda_0 t}.
    \end{align*}
The last inequality comes from Proposition 1 of Ahmed et al \cite{ahmed2026}. As long as $q\lambda_0/\lambda_1>1$, we then obtain:
\begin{align*}
    \sum_{j=1}^\infty\mathbb{E}\left[\left(\int_0^t \overline{e}_j(t-s)(av+\mu)Z_0(s)\,ds\right)^{q}\right]\lesssim e^{q\lambda_0t}.
\end{align*}
This proves the claim for $\Gamma_{1,j}$. For bound \ref{eq:g2_bound_t}, write:
\begin{align*}
    &H_{i,j}(u)=j\int_0^u \mathbf{1}_{Z_1^{(i)}(s)=j}\,ds,\,\overline{h}_j(u)=\mathbb{E}[H_{i,j}(u)],\\
    &D_j(t)=\sum_{T_i\leq t}(H_{i,j}(t-T_i)-\overline{h}_j(t-T_i)),\\
    &\Gamma_{2,j}(t)=(a+b)D_{j-1}(u)+dD_{j+1}(t)-(a+b+d)D_j(t).
\end{align*}
$H_{i,j}$ is meant to take the role of $E_{i,j}$, $\overline{h}_j$ is meant to take the role of $\overline{e}_j$, and so forth. It suffices to show:
\begin{align*}
    \sum_{j=1}^\infty\mathbb{E}[\sup_{t'\leq t}|D_{j+i}(t')|^{p}]\lesssim e^{(p/2)\lambda_0 t}
\end{align*}
for $i\in \{-1,0,1\}$. This bound follows from very similar argumentation as for $\Gamma_{1,j}$ with the following differences: (1) the very first application of the Burkholder-Davis inequality is skipped and (2) the first application of the Rosenthal inequality is replaced by Lemma \ref{lem:rosenthal_maximal}.

For the tail bounds \ref{eq:g1_bound_R} and \ref{eq:g2_bound_R}, it suffices to show that the bound on \ref{eq:e_sum_bound} can afford a decay in $R$ when replacing the sum from $j=1$ to $\infty$ with a sum from $j=R$ to $\infty$, for which it suffices to show:
\begin{align*}
    \sum_{j=R}^\infty \overline{e}_j(t)\leq CL(R)e^{ct},
\end{align*}
where $L(R)\to 0^+$ as $R\to\infty$. Let $E^T(t)$ denote the total number of events that have occurred to a type-1 clone (started at time $0$) by time $t$. Then:
\begin{align*}
    \overline{e}_j(t)\leq \mathbb{E}[E^T(t)\mathbf{1}_{Z_1^{(1)}(t)\geq j}]\leq \sqrt{\mathbb{E}[E^T(t)^2]\frac{\mathbb{E}[Z_1^{(1)}(t)^3]}{j^3}}\leq Cj^{-3/2}e^{\lambda_0(5/2)t}.
\end{align*}
Hence, letting $L(R)=R^{-1/2}$ and $c=5/2$ concludes the result. 
\end{proof}

\end{document}